\documentclass[11pt]{article}

\usepackage[utf8]{inputenc}
\usepackage[dvipsnames]{xcolor}
\usepackage{url}            
\usepackage{booktabs}       
\usepackage{amsfonts}       
\usepackage{nicefrac}       
\usepackage{mathrsfs}
\usepackage{graphicx}
\usepackage{amsmath,amssymb,amsthm}
\usepackage[sans]{dsfont}
\usepackage[square,numbers]{natbib}

\usepackage{algorithm,algorithmic}
\usepackage{thmtools}
\usepackage{thm-restate}

\usepackage[pdfencoding=auto, psdextra, colorlinks,citecolor=blue!90!black]{hyperref} 
\usepackage{cleveref}

\declaretheorem[name=Theorem,numberwithin=section]{theorem}

\declaretheorem[name=Lemma,sibling=theorem]{lemma}
\declaretheorem[name=Proposition,sibling=theorem]{proposition}

\declaretheorem[name=Corollary,sibling=theorem]{corollary}
\declaretheorem[name=Example,sibling=theorem]{example}
\declaretheorem[name=Remark,sibling=theorem]{remark}
\declaretheorem[name=Definition,sibling=theorem]{definition}

\declaretheorem[name=Assumption,sibling=theorem]{assumption}

\usepackage{bm}

\def\sgn{\textnormal{sgn}}

\DeclareMathOperator*{\argmin}{argmin}
\DeclareMathOperator{\cl}{cl}

\DeclareMathOperator{\con}{con}

\DeclareMathOperator{\dist}{dist}

\DeclareMathOperator{\exs}{exs}

\DeclareMathOperator{\gph}{gph}

\DeclareMathOperator{\nt}{int}

\DeclareMathOperator{\lip}{lip}
\DeclareMathOperator{\glip}{g-lip}

\newcommand{\reals}{\mathbb{R}}
\newcommand{\Reals}{\overline{\mathbb{R}}}
\newcommand{\natnums}{{{\rm l} \kern -.13em {\rm N} }}
\newcommand{\nats}{\mathbb{N}}
\newcommand{\Ex}{\mathbb{E}}
\newcommand{\setd}{{ d \kern -.15em l}}

\newcommand{\nsum}{\mathop{\sum}\nolimits}

\renewcommand{\liminf}{\mathop{\rm liminf}}
\renewcommand{\limsup}{\mathop{\rm limsup}}

\newcommand{\ninf}{\mathop{\rm inf}\nolimits}
\newcommand{\nsup}{\mathop{\rm sup}\nolimits}

\newcommand{\nnmin}{\mathop{\rm minimize}}

\newcommand{\nargmin}{\mathop{\rm argmin}\nolimits}

\newcommand{\lto}{\,{\lower 1pt\hbox{$\rightarrow$}}\kern -10pt
     \hbox{\raise 4pt \hbox{$\, \scriptstyle l$}}\hskip7pt}
\newcommand{\eto}{\,{\lower 1pt\hbox{$\rightarrow$}}\kern -10pt
     \hbox{\raise 4pt \hbox{$\, \scriptstyle e$}}\hskip7pt}
\newcommand{\hto}{\,{\lower 1pt\hbox{$\rightarrow$}}\kern -11pt
     \hbox{\raise 4pt \hbox{$\, \scriptstyle h$}}\hskip7pt}
\newcommand{\pto}{\,{\lower 1pt\hbox{$\rightarrow$}}\kern -11pt
     \hbox{\raise 4.5pt \hbox{$\, \scriptstyle p$}}\hskip7pt}
\newcommand{\cto}{\,{\lower 1pt\hbox{$\rightarrow$}}\kern -11pt
     \hbox{\raise 4pt \hbox{$\, \scriptstyle c$}}\hskip7pt}
\newcommand{\gto}{\,{\lower 1pt\hbox{$\rightarrow$}}\kern -11pt
     \hbox{\raise 4.5pt \hbox{$\, \scriptstyle g$}}\hskip7pt}
\newcommand{\sto}{\,{\lower 1pt\hbox{$\rightarrow$}}\kern -10pt
     \hbox{\raise 4pt \hbox{$\, \scriptstyle s$}}\hskip7pt}
     
     \def\Nto{\,{\raise 1pt\hbox{$\rightarrow$}}\kern -13pt
     \hbox{\lower 3pt \hbox{$\, \scriptstyle N$}}\hskip7pt}
\def\Cto{\,{\raise 1pt\hbox{$\rightarrow$}}\kern -14pt
     \hbox{\lower 3pt \hbox{$\, \scriptstyle C$}}\hskip7pt}
\def\fto{\,{\raise 1pt\hbox{$\rightarrow$}}\kern -14pt
     \hbox{\lower 3pt \hbox{$\, \scriptstyle f$}}\hskip7pt}

\def\dd{\textnormal{d}}

\title{Lipschitzian SLLNs for Random Functions}

\author{Lai~Tian\thanks{Daniel J.~Epstein Department of Industrial and Systems Engineering, University of Southern California, Los Angeles, CA. Emails: {\tt \{laitian, royset\}@usc.edu}} \and Johannes O.~Royset\footnotemark[1]}

\date{October 1, 2026}

\usepackage{fullpage}

\usepackage{multirow}

\def\epsilon{\varepsilon}

\newcommand{\bit}{\mathop{\textnormal{bit}}\nolimits}
\newcommand{\ball}{\mathbb{B}}
\newcommand{\compl}{\mathsf{c}}

\usepackage{bbm}
\usepackage{dsfont}

\usepackage{framed}

\def\Xto{\,{\raise 1pt\hbox{$\rightarrow$}}\kern -12pt
     \hbox{\lower 3pt \hbox{$\, \scriptscriptstyle X$}}\hskip7pt}

\def\eval{\textnormal{ev}}

\def\R-alg{\mathcal{R}}
\def\Lip{\operatorname{Lip}}

\usepackage{enumitem}

\def\VC{\textnormal{\textsf{VC}}}

\newcommand{\Enu}[1]{E^\nu\hspace{-0.09cm}#1}
\newcommand{\Ef}[1]{E\hspace{-0.05cm}#1}

\usepackage{tikz}

\definecolor{gblue}{RGB}{0,92,153}
\definecolor{gred}{RGB}{190,55,45}
\definecolor{gblack}{RGB}{45,45,45}

\begin{document}
\maketitle
\begin{abstract}
We prove strong laws of large numbers for random functions in the Lipschitz pseudometric. Our results hold under either a topological or a model-theoretic condition, with the latter encompassing functions jointly definable in o-minimal structures but extending substantially beyond this class. Applications include uniform convergence of limiting and Clarke subdifferentials and finite-sample identification of solutions. Consequently, we identify broad classes of functions for which the failure phenomena revealed by our previous negative results [Tian and Royset, arXiv:2511.16568, 2025] do not occur.
\end{abstract}

\section{Introduction}
In stochastic optimization and statistical learning, empirical approximations are expected to preserve the landscape and properties of the population problem, at least in a limiting sense.
 This principle is well understood for global solutions and, in nonconvex smooth settings, for stationary points. Our goal is to develop an analogous consistency theory for stationary points of nonconvex nonsmooth problems.

Let $(\Xi,\mathcal{A})$ be a measurable space and let $\bm{\xi},\bm{\xi}^1,\bm{\xi}^2,\ldots$ be independent and identically distributed (iid) $\Xi$-valued random variables on a complete probability space $(\Omega, \mathcal{F}, \mathbb{P})$. Throughout, we use boldface notation for random variables. For each $\nu\in\nats=\{1, 2, \dots \}$, consider the problems
\begin{align}
\nnmin\nolimits_{x\in X} \Ef{f}(x)  &= \Ex[f(\bm{\xi},x)],
\label{eq:primal}\\
\nnmin\nolimits_{x\in X} \Enu{f}(x) &= \tfrac{1}{\nu}\nsum_{i=1}^\nu
                       f(\bm{\xi}^i,x).
\label{eq:saa}
\end{align}
Here, $f:\Xi\times\reals^d\to\reals$ is $\mathcal{A}$-measurable in $\xi$ for every $x\in\reals^d$, and each slice $f(\xi,\cdot)$ is a possibly nonconvex and nonsmooth Lipschitz function on $X\subset\reals^d$.
We refer to \eqref{eq:primal} as the \emph{population} problem and to \eqref{eq:saa} as its \emph{sample-average approximation} (SAA). Since the distribution of $\bm{\xi}$ in the population problem \eqref{eq:primal} is typically unknown, a standard approach is to solve the SAA problem \eqref{eq:saa} and use its solutions as approximations to those of \eqref{eq:primal}. A fundamental question is therefore whether the solutions of \eqref{eq:saa} provide consistent approximations to those of \eqref{eq:primal} as the sample size increases.
Epigraphical and uniform strong laws of large numbers (SLLNs) establish consistency of global minimizers for the SAA problem; see, e.g., \cite{artstein1995consistency} and \cite[Section~9.2.6]{shapiro2021lectures}. 
Without convexity, computing a global minimizer of $\Enu{f}$ can be computationally costly, if feasible at all. 
A natural alternative is to consider stationary points, for which uniform SLLNs for gradients yield an analogous consistency theory in the smooth nonconvex setting; see, e.g., \cite[p.~450]{shapiro2021lectures} and \cite{mei2018landscape,foster2018uniform}. When $f(\xi,\cdot)$ may also be nonsmooth, however, even the appropriate formulation of consistency for stationary points requires additional care.

For ease of exposition, assume temporarily that $X=\reals^d$ and consider stationarity defined through the Clarke subdifferential mapping $\overline\partial \Ef{f}$; see, e.g., \cite[8(32)]{VaAn}. 
Under mild integrability and Lipschitz conditions, $\Ef{f}$ is locally Lipschitz and hence $\overline{\partial}\Ef{f}(x)$ is well defined. A point $x\in\reals^d$ is Clarke stationary for the population problem if $0\in\overline{\partial}\Ef{f}(x)$, a necessary condition for local optimality. The analogous notion for the SAA problem is obtained by replacing $\Ef{f}$ with $\Enu{f}$.
The desired conclusion is that, for a sequence of random points $\{\bm{x}^\nu\}_\nu$ and $\bm{x}$, if, $\mathbb{P}$-almost surely ($\mathbb{P}$-a.s.), $\bm{x}^\nu\to \bm{x}$ and $0 \in \overline{\partial}\Enu{f}(\bm{x}^\nu)$, then $0\in\overline{\partial}\Ef{f}(\bm{x})$, $\mathbb{P}$-a.s. For simplicity, we focus here on exact stationarity; approximate stationarity will be addressed in \Cref{sec:subd}. Most existing consistency results do not address the two set-valued mappings $\overline{\partial}\Enu{f}$ and $\overline{\partial}\Ef{f}$; see, e.g., \cite{shapiro2007uniform,xu2010uniform,norkin-wets,milz2025consistency}. Instead, these works interchange expectation and subdifferentiation and study $\frac{1}{\nu}\sum_{i=1}^\nu\overline{\partial}_x f(\bm{\xi}^i,\cdot)$ and $\Ex[\overline{\partial}_x f(\bm{\xi},\cdot)]$, where $\overline{\partial}_x f(\xi,\cdot)=\overline{\partial}(f(\xi,\cdot))$ denotes the partial Clarke subdifferential and the expectation is understood in the Aumann sense \cite{aumann1965integrals}. The solution concept defined by the zero sets of $\frac{1}{\nu}\sum_{i=1}^\nu\overline{\partial}_x f(\bm{\xi}^i,\cdot)$ and $\Ex[\overline{\partial}_x f(\bm{\xi},\cdot)]$ is commonly referred to as \emph{weak stationarity}; see, e.g., \cite{xu2010uniform,burke2020subdifferential}. Consistency of weak stationarity typically holds under mild integrability and Lipschitz assumptions \cite{shapiro2007uniform, norkin-wets,xu2010uniform}. By standard interchange rules \cite[Theorem 2.7.2]{clarke1990optimization}, one has $\overline{\partial}\Enu{f}(x)\subset\frac{1}{\nu}\sum_{i=1}^\nu\overline{\partial}_x f(\bm{\xi}^i,x)$ and $\overline{\partial}\Ef{f}(x)\subset\Ex[\overline{\partial}_x f(\bm{\xi},x)]$, while the reverse inclusions generally fail unless suitable subdifferential regularity conditions hold \cite[Definition~7.25]{VaAn}. Thus, \emph{stationarity}, defined through the Clarke subdifferential of the expectation function, is a sharper solution concept than weak stationarity.

Indeed, this distinction between weak stationarity and stationarity can be decisive. For illustration, let $\phi:\reals\to\reals$ be a $1$-Lipschitz function satisfying $\overline{\partial}\phi(x)=[-1,1]$ for every $x\in\reals$; see  \cite[p.~129]{rockafellar1982favorable} for a construction. Let $\Xi=\{-1,1\}$ with equal probabilities, and define $f(\xi,x)=\tfrac12x+\xi\phi(x)$. Then $\overline{\partial}\Ef{f}(x)=\{\tfrac12\}$ and $\overline{\partial}_x f(\xi,x)=[-\tfrac12,\tfrac32]$ for every $x\in\reals$ and $\xi \in \Xi$. Thus, any $x\in\reals$ is weakly stationary for both the population and SAA problems, because $0\in\Ex[\overline{\partial}_x f(\bm{\xi},x)]=\frac{1}{\nu}\sum_{i=1}^\nu\overline{\partial}_x f(\bm{\xi}^i,x)$ for any $\nu \in \nats$. But the population objective $\Ef{f}$  has no stationary point. Consequently, consistency of weak stationarity may hold under mild assumptions yet remain vacuous. By contrast, consistency of stationarity is a stronger concept but can fail under standard assumptions; see \cite[Section~2.3]{tian2025failure}.

The desired conclusion can be sharpened further by using finer notions of subdifferentiation. For a locally Lipschitz function $h:\reals^d\to\reals$, the Clarke subdifferential satisfies $\overline{\partial}h(x)=\con\partial h(x)$, where $\partial h$ denotes the limiting  (Mordukhovich) subdifferential; see, e.g., \cite[Definition~8.3]{VaAn}. 
We call $x$ a limiting stationary point of $h$ if $0 \in \partial h(x)$, which is also a necessary condition for local optimality. The convexification defining the Clarke subdifferential may produce Clarke stationary points that are not limiting stationary. For example, $\partial(-|\cdot|)(0)=\{-1,1\}$, while $\overline{\partial}(-|\cdot|)(0)=[-1,1]$. Thus, the strict local maximizer $0$ is Clarke stationary but not limiting stationary. In general, limiting stationarity is a sharper solution concept than Clarke stationarity for minimization problems. They are equivalent when $h$ is subdifferentially regular \cite[Corollary~8.11]{VaAn}. 
Notably, the advantage of limiting stationarity can be lost when one considers weak stationarity, since $\Ex[\partial f(\bm{\xi},x)] = \Ex[\overline{\partial} f(\bm{\xi},x)]$ whenever $\bm{\xi}$ has an atomless distribution  \cite{aumann1965integrals}.

In this paper, we establish SAA consistency results for nonconvex nonsmooth problems that are sharper in two respects: We compare the subdifferentials of the population and SAA objectives without interchanging expectation and subdifferentiation, and we treat limiting subdifferentials as well as their Clarke convexifications. We achieve this by establishing new uniform convergence results for limiting and Clarke subdifferentials; see \Cref{coro:ulln-subd}. Existing results concern enlarged Clarke subdifferentials \cite{shapiro2007uniform}, or are restricted to subdifferentially regular functions  for which the limiting and Clarke subdifferentials coincide; see \cite{shapiro2007uniform,davis2022graphical,ruan2024subgradient,milz2025consistency}. To the best of our knowledge, no previous uniform law compares $\partial \Enu{f}$ and $\partial \Ef{f}$ for a class of locally Lipschitz random functions that need not be subdifferentially regular. Therefore, no comparable consistency theory has been available for limiting and Clarke stationary points.

Our approach is to compare the population and SAA functions before passing to subdifferentials. 
Specifically, we establish the following \emph{Lipschitzian SLLN}:
\begin{equation}\label{eq:lip-slln}
d_X(\Enu{f},\Ef{f})\to0,
\qquad
\mathbb{P}\text{-a.s.},
\end{equation}
where, for $X\subset\reals^d$ with $0 \in X$, $d_X(h,g)$ denotes the Lipschitz pseudometric
\[
d_X(h,g)
=
\max\bigl\{|h(0)-g(0)|,\glip_X(h-g)\bigr\},
\]
and $\glip_X f$ denotes the global Lipschitz modulus of $f$ on $X$; see \Cref{sec:prel-lip} for the formal definitions.
The pseudometric $d_X(\Enu{f}, \Ef{f})$ controls both the function values and the global Lipschitz modulus of the function $\Enu{f} - \Ef{f}$ on the set $X$, making \eqref{eq:lip-slln} a strong mode of convergence sufficient for our purposes. Indeed, if $h$ and $g$ are locally Lipschitz on an open set $X \subset \reals^d$, then
\begin{equation}\label{eq:intro-subd-bound}
\nsup_{x\in X}
\setd\bigl(\overline{\partial}h(x),\overline{\partial}g(x)\bigr)
\leq
\nsup_{x\in X}
\setd\bigl(\partial h(x),\partial g(x)\bigr)
\leq
d_X(h,g),
\end{equation}
where $\setd(A,B)$ denotes the Hausdorff distance between nonempty sets $A$ and $B$; see \Cref{prop:subd} below.
Thus, we obtain simultaneous uniform control of limiting and Clarke subdifferentials  and, consequently, the desired consistency of stationary points.  Nevertheless, our previous work \cite{tian2025failure} shows that the uniform law in terms of subdifferentials in \eqref{eq:intro-subd-bound}, with $h=\Enu{f}$ and $g = \Ef{f}$, cannot hold in general even when every slice $f(\xi, \cdot)$ is convex and the function $f$ satisfies the standard integrability and Lipschitz assumptions. Hence, additional assumptions are necessary for the Lipschitzian law in \eqref{eq:lip-slln} to hold.

We prove \eqref{eq:lip-slln} under either of two new structural conditions, together with mild integrability and Lipschitz conditions. The first is a topological separability condition; see \Cref{as:separable}. It holds, in particular, when either the distribution is discrete or the function admits a stochastic linear-combination representation in terms of deterministic Lipschitz functions; see \Cref{sec:sep}. The second is a model-theoretic definability condition; see \Cref{as:definable}. It encompasses many nonsmooth functions arising in practice \cite{ioffe2009invitation}, including losses from modern learning models \cite{bareilles2025deep}; see \Cref{sec:def}.
For illustration, under mild integrability and Lipschitz assumptions, the Lipschitzian law \eqref{eq:lip-slln} holds if \emph{either} of the following conditions is satisfied (cf.~\Cref{thm:lln,thm:o-minimal}):
\smallskip\smallskip
\begin{enumerate}[label=\textnormal{(\roman*)}]
\item The family $\{f(\xi,\cdot)|_X\mid\xi\in\Xi\}$ is essentially separable in the Lipschitz space $\Lip(X)$.\footnote{See \Cref{sec:prel-lip} for the definition of ``Lipschitz space.'' }
\item For almost every $\xi\in\Xi$, the slice $f(\xi,\cdot)$ is definable in a fixed countable-language o-minimal structure.\footnote{See \Cref{sec:pre-model} for the definition of ``countable-language o-minimal structure.'' } For example, $f(\xi, \cdot)$ is semialgebraic for almost every $\xi \in \Xi$.
\end{enumerate}\smallskip\smallskip
Despite their different nature, both conditions ensure the Lipschitzian law in \eqref{eq:lip-slln} and thereby rule out the failure phenomena exhibited in \cite{tian2025failure}.

In addition to establishing the Lipschitzian law \eqref{eq:lip-slln}, we study its convergence rate, obtaining the canonical rate $O_{\mathbb{P}}(\nu^{-1/2})$ in one setting and slow convergence results in others. Beyond uniform convergence of subdifferentials, we also apply \eqref{eq:lip-slln} to identify solutions to the population problem using only a finite number of samples, extending existing results \cite{shapiro2000rate,kleywegt2002sample,shapiro2002conditioning} by relaxing convexity, polyhedrality, and distributional assumptions.

The remainder of the paper is organized as follows. \Cref{sec:prel} reviews Lipschitz spaces, VC classes, and the model-theoretic notions used throughout. \Cref{sec:lslln} introduces the two structural assumptions and establishes the Lipschitzian SLLNs and their convergence rates. \Cref{sec:verify} provides practical conditions ensuring these two structural assumptions. \Cref{sec:appl} discusses applications and related work, including uniform SLLNs for subdifferentials and finite-sample identification of solutions. Finally, \Cref{sec:conclusion} gives concluding remarks.

\paragraph{Notation.}
We use mostly standard notation. The sets of positive integers, integers, and rational numbers are denoted by $\nats=\{1,2,\ldots\}$, $\mathbb{Z}$, and $\mathbb{Q}$, respectively. For $\nu\in\nats$, write $[\nu]=\{1,\ldots,\nu\}$, and let $\Reals=\reals\cup \{-\infty,\infty\}$. For a set $S$, we denote its cardinality by $|S|$ and its power set by $2^S$. For disjoint sets $A$ and $B$, we write the union $A\cup B$ as $A\sqcup B$.
Let $\|\cdot\|$ denote the Euclidean norm, and $\ball=\{y\mid \|y\|\leq 1\}$. A function $g:\reals^d\to\reals$ is $L$-Lipschitz on $X\subset\reals^d$ if $|g(x)-g(y)|\leq L\|x-y\|$ for all $x,y\in X$, and simply $L$-Lipschitz when $X=\reals^d$.  
We denote the restriction of $g$ to $X$ by $g|_X$ and the graph of $g$ by $\gph g$.
For nonempty sets $C,D\subset\reals^d$ and $x\in\reals^d$, define $\dist(x,C)=\inf_{z\in C}\|x-z\|$. The excess of $C$ relative to $D$ is $\exs(C;D)=\sup_{z\in C}\dist(z,D)$, and the Hausdorff distance between $C$ and $D$ is $\setd(C,D)=\max\{\exs(C;D),\exs(D;C)\}$. The convex hull and real linear span of a set $S$ are denoted by $\con S$ and $\operatorname{span}S$, respectively.  
For a set $C$, define $\mathbf{1}_C(x)=1,\iota_C(x)=0$ if $x\in C$, and $\mathbf{1}_C(x)=0,\iota_C(x)=\infty$ otherwise. 
 For $t\in\reals$, define $\sgn(t)=t/|t|$ if $t\neq0$ and $\sgn(t)=0$ otherwise.
The Borel $\sigma$-algebra on a topological space $S$ is denoted by $\mathcal{B}(S)$. 
A sequence of random variables $\{\bm{y}^\nu\}_{\nu\in\nats} = O_{\mathbb{P}}(r^\nu)$ if $\{\bm{y}^\nu/r^\nu\}_{\nu\in\nats}$ is bounded in probability. 

\section{Preliminaries}\label{sec:prel}
In this section, we set the stage and recall concepts and results used in the sequel. Let $(\Xi,\mathcal{A})$ be a measurable space, and let $\bm{\xi},\bm{\xi}^1,\bm{\xi}^2,\ldots$ be iid $\Xi$-valued random variables.  For an $\mathcal{A}$-measurable function $h:\Xi \to \reals$, when the underlying random variable is clear from the context, we write its expectation $\Ex[h(\bm{\xi})]$ as $\Ex h$ and, by a slight abuse of notation, denote the sample average $\frac{1}{\nu}\sum_{i=1}^\nu h(\bm{\xi}^i)$ by $\Ex^\nu[h(\bm{\xi})]$ or $\Ex^\nu h$.

For a function $f:\Xi\times\reals^d\to\reals$ such that $f(\cdot,x)$ is $\mathcal{A}$-measurable for every $x\in\reals^d$, consider the expectation function $\Ef{f}:\reals^d\to\Reals$ and its sample-average approximation $\Enu{f}:\reals^d\to\reals$ defined in \eqref{eq:primal} and \eqref{eq:saa},
adopting the convention $\infty- \infty = \infty$ if needed. 
As $\nu \to \infty$, we are interested in the convergence of $\Enu{f}$ to $\Ef{f}$ in suitable topologies and its consequences. Classical examples include epigraphical convergence \cite[Section 7B]{VaAn} and uniform convergence \cite[Section 7C]{VaAn}. 
Results of this kind are referred to in the literature as epigraphical SLLNs \cite{artstein1995consistency} and uniform SLLNs \cite[Section~9.2.6]{shapiro2021lectures}, which form the cornerstone of consistency theory for global minimizers.

In this paper, we focus on a stronger mode of convergence, formulated through the Lipschitz space and the associated Lipschitz pseudometric, which we recall next.

\subsection{Lipschitz space, norm, and pseudometric}\label{sec:prel-lip}
For a real-valued function $h:Y \to \reals$ and nonempty sets $X\subset Y \subset \reals^d$, the \emph{global Lipschitz modulus} $\glip_X h \in \Reals$ of the function $h$ on the set $X$ is
\[
\glip_X h=\sup_{x,y \in X,x \neq y} \frac{h(x) - h(y)}{\|x - y\|},
\]
if $X$ is not a singleton and  $\glip_X h=0$ otherwise. We may write $\glip_{X} h$ as  $\glip h$ when $X$ is the domain of definition $Y$ for $h$. 
The \emph{local Lipschitz modulus} of $h$ at $x$ is defined by \cite[Definition 9.1]{VaAn}
\[
\lip h(x)
=
\limsup_{y,z\to x,\ y\neq z}
\frac{h(y)-h(z)}{\|y-z\|}.
\]
	If $X\subset\reals^d$ and $h:\reals^d\to\reals$, then for every $x\in\nt X$, $\lip h(x) \leq \glip_X h$.

When $0 \in X$, consider the \emph{space of Lipschitz functions} $\Lip(X)$ on $X$ defined as 
\[
    \Lip(X)
    =
    \{h:X\to\reals \mid \glip h < \infty\},
\]
which is endowed with the Lipschitz norm (see \cite{beer2013lipschitz} and \cite[Section 4]{beer2023bornologies}):
\[
\|h\|_{\lip} = \max\{|h(0)|, \glip h\}.
\]
The space $(\Lip(X), \|\cdot \|_{\lip})$ is a Banach space \cite[Proposition 4.1]{beer2023bornologies}. The point $0 \in X$ in the definition above can be replaced by any point in $X$ without changing the topology \cite[p.~27]{beer2023bornologies}. We make such a choice for notational simplicity. For $0 \in X \subset \reals^d$ and functions $h,g:\reals^d\to\Reals$ that are real-valued on $X$, 
the Lipschitz norm $\|\cdot\|_{\lip}$ on $\Lip(X)$ also gives rise to a Lipschitz (extended) pseudometric between $h$ and $g$ as follows:
\[
d_{X}(h, g) = \|(h - g)|_X\|_{\lip} = \max\{|h(0) - g(0)|, \glip{}(h|_X-g|_X)\}.
\] 
When $(h-g)|_X \in \Lip(X)$, one has $d_X(h,g) < \infty$.
\subsection{VC-classes of sets and functions}

Let $\mathcal{C}$ be a collection of subsets of a set $S$ and $\{x_1, \ldots, x_n\}$ be an arbitrary finite subset of $S$.  We say the  collection $\mathcal{C}$ \emph{picks out} a subset $A$ from  $S$ if there exists $C \in \mathcal{C}$ such that $A = C \cap \{x_1, \ldots, x_n\}$. The collection $\mathcal{C}$ is said to \emph{shatter} $\{x_1, \ldots, x_n\}$ if each of its $2^n$ subsets can be picked out. The \emph{Vapnik--Chervonenkis dimension} (VC-dimension) $\VC(\mathcal{C})$ of the class $\mathcal{C}$ is the largest $n \in \nats\cup \{0\}$ such that some $\{x_1, \ldots, x_n\}\subset S$ can be shattered by $\mathcal{C}$; see \cite[Section 2.6.1]{vanderVaartWellner.23}. 
If $\mathcal{C}$ shatters sets of arbitrarily large finite size, we set $\VC(\mathcal{C})=\infty$.
If $\VC(\mathcal{C}) < \infty$, we call $\mathcal{C}$ a VC-class.  

VC-classes are stable under certain operations; see \cite[Section~2.6.6]{vanderVaartWellner.23}. The following properties will be useful for us.

\begin{lemma}[permanence]\label{lem:vc-rules}
	Given $G:X \to Y$ and VC-classes $\mathcal{C},\mathcal{D} \subset 2^Y$, the following hold. 
	\begin{enumerate}[label=\textnormal{(\alph*)}]
		\item The family $\{C \cup D \mid C \in \mathcal{C}, D \in \mathcal{D}\}$ is a VC-class; see \cite[Lemma 2.6.19(iii)]{vanderVaartWellner.23}.
		\item The family $\mathcal{C}\cup \mathcal{D}$ is a VC-class; see \cite[Exercise 6.11]{shalev2014understanding}.
		\item The family $G^{-1}(\mathcal{C})=\{G^{-1}(C) \mid C \in \mathcal{C}\}$ is  a VC-class; see \cite[Lemma 2.6.19(v)]{vanderVaartWellner.23}.
	\end{enumerate}
\end{lemma}

We define similar concepts for functions by considering their subgraphs (or hypographs); see \cite[Section 2.6.2]{vanderVaartWellner.23}.
For a collection of functions $\mathcal{H}=\{h_i:\Xi \to \reals \mid i \in I\}$, we call $\mathcal{H}$ a VC-subgraph class if the collection $\{\{(\xi, r) \in \Xi \times \reals \mid r < h_i(\xi)\} \mid i \in I\}$ is a VC-class.

\subsection{Definability, NIP, and o-minimality}\label{sec:pre-model}
In model theory (see, e.g., \cite{chang1990model} and also \cite[Section 1.5]{van1998tame}), a first-order language
 $\mathcal L$ consists of relation, function, and constant symbols. 
An $\mathcal L$-structure $\mathcal M$ consists of a nonempty universe $M$ together with interpretations of all symbols of $\mathcal L$ on $M$. 
When $\mathcal L$ is countable, we call $\mathcal{M}$ a \emph{countable-language} structure.
An $\mathcal L$-formula is built from atomic formulas using logical connectives and quantifiers. 
 A set
$S \subset M^d$ is said to be \emph{definable} in the structure $\mathcal{M}$ if there exist an $\mathcal{L}$-formula $
\varphi(x;y),$
with variables $x=(x_1,\ldots,x_d)$ and $y=(y_1,\ldots,y_m)$, and parameter $b\in M^m$ such that
\[
S=\varphi(M^d;b)
=\{a\in M^d \mid \mathcal{M}\models\varphi(a;b)\}.
\]
Here, $\mathcal{M}\models\varphi(a;b)$ means that the formula $\varphi(a;b)$ holds in $\mathcal{M}$. We may simply write $\varphi(a;b)$ when $\mathcal{M}$ is clear from context.
The semicolon in $\varphi(x;y)$ has no syntactic significance; it is used
only to distinguish the object variables $x$ from the parameter variables
$y$.
To emphasize the dependence, we may say that the set $S$ is definable in the structure $\mathcal{M}$ by the formula $\varphi$ with parameter $b$. 
We call a collection $\{S_i\subset M^d \mid i \in I\} $ \emph{uniformly definable} (by $\varphi$) if, for any $i \in I$, one has $S_i = \varphi(M^d; b_i)$ for some $b_i \in M^m$.
An extended-real-valued function is called definable if its graph is definable. Let $\mathcal{M}$ and $\mathcal{M}'$ have the same universe. If the language of $\mathcal{M}'$ extends that of $\mathcal{M}$ and the interpretations of the common symbols agree, then $\mathcal{M}'$ is an \emph{expansion} of $\mathcal{M}$, and $\mathcal{M}$ is a \emph{reduct} of $\mathcal{M}'$. In this paper, we will only consider structures that are expansions of the real ordered field $\R-alg=(\reals, +, \cdot, <, 0, 1)$.

For illustration, we examine $\R-alg$ more closely. The structure $\R-alg$ has universe $\reals$, function symbols $\{+,\cdot\}$, the relation symbol $\{<\}$, and constant symbols $\{0,1\}$. Accordingly, its language $\mathcal{L}$ consists of the symbols $(+,\cdot,<,0,1)$, interpreted in the usual way. A formula in that language $\mathcal{L}$ is built inductively from atomic formulas using the logical connectives $\neg, \wedge, \vee$, together with the quantifiers $\exists, \forall$. For example, define a formula $\varphi(x;y) = (x_1 > y) \wedge (x_2 > y)$, where $x=(x_1,x_2) \in \reals^2$ is the object variable and $y \in \reals$ is the parameter variable. For every $b \in \reals$, one has $\varphi(\reals^2; b) = (b,\infty)^2$. Thus, a single formula $\varphi$ defines a family of sets as the parameter $b$ varies. Notably, by the Tarski--Seidenberg theorem \cite[Theorem 2.6]{coste2002introduction}, the sets definable in the structure $\R-alg$ are precisely the \emph{semialgebraic} sets, namely, finite Boolean combinations of sets defined by polynomial equalities and inequalities with real coefficients; see \cite[Section 2.1.1]{coste2002introduction}.

Functions definable in an arbitrary expansion of $\R-alg$ can be quite wild. We will focus on expansions that have the \emph{non-independence property} (NIP); see \cite{simon2014guide} for a systematic introduction.
\begin{definition}[{\cite[Definitions~1.1 and 1.2]{laskowski1992vapnik}}]
	A formula $\phi(x_1,\ldots, x_m; y_1, \cdots, y_k)$ has the independence property (IP) with respect to a structure $\mathcal{M}$ with universe $M$ if for every $n \in \nats$ there is a collection $a_1, \ldots, a_n$ of elements from $M^k$ so that for every $I \subset [n]$ there is $c_I \in M^m$ with
	\[
	\mathcal{M} \models \phi(c_I; a_i) \quad \iff \quad i \in I.
	\]
A formula has IP if it has the independence property, and has NIP otherwise. A structure $\mathcal{M}$ has NIP if every formula in its language has NIP with respect to $\mathcal{M}$; otherwise, it has IP.
\end{definition}
Our focus is on functions and sets definable in NIP expansions of $\R-alg$. A key consequence of such definability is the following observation.

\begin{lemma}[{\cite[Proposition 1.3]{laskowski1992vapnik}}]\label{lem:nip-vc}
	For a structure $\mathcal{M}$ with universe $M$ and a formula $\phi(x; y)$ with $x \in M^m$ and $y \in M^k$ in the language of $\mathcal{M}$, the family 
	\[
	\big\{ \{c \in M^m \mid \mathcal{M} \models \phi(c; b)\} \bigm| b \in M^k\big\}
	\]
	is a VC-class if and only if the formula $\phi$ has NIP.
\end{lemma}

VC-classes play a central role in statistical learning theory \cite[Chapter~6]{shalev2014understanding}. Motivated by this connection, definability in NIP structures has been studied extensively in learning theory \cite{macintyre1993finiteness,karpinski1997polynomial,livni2013honest,chase2019model} and, more recently, in stochastic programming \cite{nguyen2026rates}. However, these works focus primarily on uniform convergence of function values, which is insufficient to control first-order objects such as subdifferentials.

Another important concept in model theory is that of an \emph{o-minimal} structure; see \cite{van1996geometric} and \cite[Definition 1.4]{coste1999introduction}. An expansion of $\R-alg$ is o-minimal if every definable subset of $\reals$ is a finite union of points and intervals.
 Such structures have been widely used in optimization to exclude pathological behavior; see \cite{ioffe2009invitation} and the references therein. The real ordered field $\R-alg$ introduced above  and the real exponential field $\mathcal{R}_{\rm exp}=(\R-alg,\exp)$ are both o-minimal structures; see \cite[Exercise 1.7]{coste1999introduction} and  \cite{wilkie1996model}.

\begin{proposition}[{\cite[Example~2.12]{simon2014guide}, \cite[p.~90, (3.15)]{van1998tame}}]\label{prop:om-NIP}
Every o-minimal structure has NIP. 
\end{proposition}

Therefore, by \Cref{lem:nip-vc}, every uniformly definable family of sets in an o-minimal structure is a VC-class; see also \cite[Chapter 5]{van1998tame} and \cite[Remark 3.42]{bareilles2025deep}.

\section{SLLNs in the Lipschitz Pseudometric}\label{sec:lslln}

In this section, we present our main results for the convergence of $\Enu{f}$ to $\Ef{f}$ in terms of the Lipschitz pseudometric. 
These results are built on structural conditions that will be examined in more detail in \Cref{sec:verify}. 
We begin with a blanket integrability and Lipschitz assumption.

\begin{assumption}[blanket]\label{as:f-S}
Given a probability space $(\Xi, \mathcal{A}, P)$,
for a function $f:\Xi \times \reals^d \to \reals$, a random variable $\bm{\xi}$ with law $P$, and a set $X \subset \reals^d$ with $0 \in X$, assume the following hold.
\begin{enumerate}[label=\textnormal{(\roman*)}]
	\item For any $x \in \reals^d$, the function $f(\cdot, x)$ is $\mathcal{A}$-measurable.
	\item $\Ex[|f(\bm{\xi}, 0)|] < \infty$.
	\item There exists an $\mathcal{A}$-measurable function $L:\Xi \to [0,\infty)$ with $\Ex L < \infty$ such that, for $P$-almost every ($P$-a.e.) $\xi \in \Xi$, the function $f(\xi, \cdot)$ is $L(\xi)$-Lipschitz  on $X$.
\end{enumerate}
\end{assumption}

\Cref{as:f-S} imposes mild conditions ensuring well-definedness; see also \cite{shapiro2007uniform,ruan2024subgradient} for similar assumptions. The condition $0\in X$ is without loss of generality. Indeed, for any $x_0\in X$, one may replace $X$ by $X-\{x_0\}$
and $f$ by the translated function $(\xi,x)\mapsto f(\xi,x+x_0)$. 
Nevertheless, \Cref{as:f-S} alone is not sufficient to guarantee the convergence $
d_X(\Enu{f}, \Ef{f}) \to 0,$
as demonstrated by the following result.

 \begin{proposition}[failure]\label{prop:counterex}
Let $X=[0,1]\subset \reals$ and $\bm{\xi}, \bm{\xi}^1,\bm{\xi}^2,\ldots$ be iid random variables on the complete probability space $(\Omega,\mathcal{F},\mathbb{P})$, each uniformly distributed on $\Xi=[0,1]$.	
There exists a function $f:\Xi\times \reals \to \reals$ such that \Cref{as:f-S} holds, and 
\[
\liminf_{\nu \to \infty} d_X(\Enu{f}, \Ef{f}) \geq \tfrac{1}{2},\qquad \mathbb{P}\text{-a.s.}
\]
\end{proposition}
\begin{proof}
	Let $f$ be the univariate construction in \cite[Theorem 1]{tian2025failure}, which satisfies \Cref{as:f-S}. 
	In particular, it can be written as follows:
	\begin{equation}\label{eq:ts25}
		f(\xi, x) = \nsum_{k=1}^\infty \bit_k(\xi)g_k(x), \qquad g_k(x)=\int_0^{\max\{x,0\}} \bm{1}_{B_k}(t)\dd t,
	\end{equation}
	where $\bit_k(\xi) \in \{0,1\}$ is the $k$th binary bit of $\xi$, $B_k = (\tfrac{1}{k}-\tfrac{1}{4k^2}, \tfrac{1}{k}+\tfrac{1}{4k^2})$, and $g_k$ is $1$-Lipschitz for any $k \in \nats$. 
	From \cite[Section 3.2, Step 5]{tian2025failure}, there exist random points $\{\bm{p}^\nu\}_\nu \subset   \{k^{-1}\}_{k \in \nats}\subset X$ such that, $\mathbb{P}$-a.s., $\Enu{f}$ and $\Ef{f}$ are both continuously differentiable in a neighborhood of  $\bm{p}^\nu$ and 
	\[
	\liminf_{\nu \to \infty} d_X(\Enu{f}, \Ef{f}) \geq \liminf_{\nu \to \infty} \|\nabla \Enu{f}(\bm{p}^\nu) - \nabla \Ef{f}(\bm{p}^\nu)\| \geq \tfrac{1}{2},
	\]
	where the last inequality follows from \cite[Theorem 1]{tian2025failure}.
\end{proof}

\subsection{Two structural assumptions}

We propose two structural conditions to overcome failures of SLLN in the Lipschitz pseudometric as illustrated by \Cref{prop:counterex}.
\subsubsection*{Separability}
Our first condition is of a topological nature, concerning separability in Lipschitz space.

\begin{assumption}[separability]\label{as:separable}
	Given a probability space $(\Xi, \mathcal{A}, P)$, $f:\Xi \times \reals^d \to \reals$, and $X \subset \reals^d$ with $0 \in X$, assume that there exists $\Xi' \subset \Xi$ with $P(\Xi') = 1$ such that $\{f(\xi, \cdot)|_X\mid \xi \in \Xi'\}$  is contained in a separable subspace of $(\Lip(X), \|\cdot\|_{\lip})$.
\end{assumption} 

Recall that separability means the existence of a countable dense subset, here with respect to the Lipschitz norm $\|\cdot\|_{\lip}$.
Under \Cref{as:f-S}, a straightforward  sufficient condition for \Cref{as:separable} is $\Xi$ being countable, which is natural when the random variable $\bm{\xi}$ has a discrete distribution. Separability, in various forms, has long been used to establish SLLNs for random variables taking values in infinite-dimensional spaces; see, e.g., \cite{ledoux1991probability,teran2008uniform,banholzer2022rates,devale2025uniform,milz2025consistency}. None of these works, however, considers the Lipschitz space $(\Lip(X),\|\cdot\|_{\lip})$ as we do here. We defer the discussion on functions satisfying \Cref{as:separable} to \Cref{sec:sep}. 

The separability requirement in \Cref{as:separable} is not mild in general. Indeed, as shown in the following example, adapted from \cite[Example 3.10]{norkin-wets}, separability can fail for a convex, piecewise affine $f:\reals^2 \to \reals$.

\begin{example}[nonseparability]\label{ex:nsep}
Let $d=1$, $X=\reals$, $\Xi=[0,1]$, $\mathcal{A}=\mathcal{B}(\Xi)$, $f(\xi,x)=|x-\xi|$, and let $P$ be the uniform distribution on $\Xi$.
Then \Cref{as:f-S} holds. Moreover, for any distinct $\xi,\xi'\in\Xi$, one has $\|f(\xi,\cdot)-f(\xi',\cdot)\|_{\lip}=2.$
Hence, for any $\Xi'\subset \Xi$ with $P(\Xi')=1$,  the family $\{f(\xi,\cdot)|_X\mid \xi\in\Xi'\}$ is not contained in any separable subspace of $\Lip(X)$.
\end{example}

\subsubsection*{Definability}

As illustrated by \Cref{ex:nsep}, \Cref{as:separable} can exclude even simple convex semialgebraic functions. We therefore introduce the following model-theoretic condition.

\begin{assumption}[piecewise uniform definability]\label{as:definable}
	Given a probability space $(\Xi, \mathcal{A}, P)$, $X \subset \reals^d$, and a function $f:\Xi \times \reals^d \to \reals$, assume that, for some dense $Q=\{q_k\}_{k \in \nats}\subset X$, some index set $I \subset \nats$, and every $\ell \in I$, there exist 
	\begin{enumerate}[label=\textnormal{(\roman*)}]
	\item a measurable set $\Xi_\ell \subset \Xi$ relative to the $P$-completion $\mathcal{A}^P$,\label{item:as:definable-meas}
		\item an NIP expansion $\mathcal{M}_\ell$ of the real ordered field $\R-alg$, and
		\item a formula $\phi_\ell(x,r; z)$ with $x \in \reals^d$, $r \in \reals$, and $z \in \reals^{k_\ell}$ in the language of $\mathcal{M}_\ell$ 
	\end{enumerate}
such that $P(\cup_{\ell \in I} \Xi_\ell) = 1$. Moreover, for any $\ell \in I$ and $\xi \in \Xi_\ell$, there exists $z \in \reals^{k_\ell}$ such that
\begin{equation}\label{eq:def-q}
	\{f(\xi, q_k)\}=\{r \in \reals \mid  \mathcal{M}_{\ell} \models \phi_\ell(q_k, r; z)\}, \quad \forall k \in \nats.
\end{equation}
\end{assumption}

Instead of separability, \Cref{as:definable} requires the family $\{f(\xi, \cdot)|_X \mid \xi \in \Xi\}$ to be definable by NIP structures $\{\mathcal{M}_\ell\}_{\ell \in I}$ in a certain way.
Roughly, for each $\ell\in I$, the family $\{\gph f(\xi,\cdot)\mid \xi\in\Xi_\ell\}$ is uniformly definable in the NIP structure $\mathcal{M}_\ell$ after restricting both the functions and the corresponding defining formulas to $Q$.
Such restriction can be useful for verifying the measurability  requirement in \ref{item:as:definable-meas}.
Recall that the $P$-completion of $\mathcal{A}$ contains all sets in $\mathcal{A}$ and also all subsets of $S \in \mathcal{A}$ if $P(S) = 0$. The set $X$ need not be definable or measurable. 

Although \Cref{as:definable} may initially appear restrictive, it naturally generalizes joint definability in a fixed o-minimal structure, which is pervasive in practice; see discussion later in \Cref{sec:j-def}. 
The countable cover $\{\Xi_\ell\}_{\ell\in I}$ broadens our framework by allowing different defining formulas and NIP structures on different pieces; see \cite{van1996new,van2000preservation} and \cite[Chapter~7]{shalev2014understanding} for related extensions of uniform convergence results.
With a suitable cover, slicewise definability suffices for \eqref{eq:def-q} under mild assumptions on the underlying structures; see \Cref{sec:pu-def}, particularly \Cref{as:p-def}.

\subsection{Main results}

When \Cref{as:f-S} and either one of Assumptions~\ref{as:separable} and \ref{as:definable} hold, the SAA function $\Enu{f}$ converges to the population function $\Ef{f}$ in the Lipschitz pseudometric.

\begin{theorem}[Lipschitzian SLLN]\label{thm:lln}
Let $(\Xi,\mathcal{A})$ be a measurable space and $\bm{\xi}, \bm{\xi}^1,\bm{\xi}^2,\ldots$ be $\Xi$-valued iid random variables on the complete probability space $(\Omega,\mathcal{F},\mathbb{P})$ with a law $P$.
Suppose that $f:\Xi\times\reals^d\to\reals$ and $X\subset\reals^d$ with $0 \in X$ satisfy Assumption~\ref{as:f-S} and that, in addition, either Assumption~\ref{as:separable} or \ref{as:definable} holds. Then 
\[
\lim_{\nu \to \infty} d_X(\Enu{f}, \Ef{f}) = 0,\qquad \mathbb{P}\text{-a.s.}
\]
\end{theorem}

As an immediate example, by considering integrands that are constant in the decision variable, \Cref{thm:lln} recovers Kolmogorov's strong law of large numbers as a special case.

\begin{example}[Kolmogorov's SLLN]
Suppose that $(\Xi,\mathcal{A}) = (\reals,\mathcal{B}(\reals))$ and $\bm{\xi}, \bm{\xi}^1,\bm{\xi}^2,\ldots$ are $\Xi$-valued iid random variables on the complete probability space $(\Omega,\mathcal{F},\mathbb{P})$ with $\Ex[|\bm{\xi}|] < \infty$. 
Let $d=1,X=\reals$, $f(\xi, x)=\xi$ for all $x \in \reals$. Hence, Assumptions~\ref{as:f-S} and \ref{as:separable} hold trivially by the separability of $\reals$, and \Cref{thm:lln} reduces to Kolmogorov's SLLN.
\end{example}

The proof of \Cref{thm:lln} will be divided into two parts and presented in \Cref{sec:prf-countable,sec:prf-def}. We describe now the basic ideas.
When Assumptions~\ref{as:f-S} and~\ref{as:separable} hold, we may regard the mapping $\xi\mapsto f(\xi,\cdot)|_X$ as a random variable taking values in a separable subspace of $(\Lip(X),\|\cdot\|_{\lip})$. This viewpoint reduces the convergence of $d_X(\Enu{f},\Ef{f})$ to an SLLN for $\Lip(X)$-valued random variables. Thus, it remains to verify the applicability of a known strong law \cite[Chapter~7]{ledoux1991probability} and to identify the resulting Bochner integral with the corresponding pointwise integral.
When Assumptions~\ref{as:f-S} and~\ref{as:definable} hold, the mapping $\xi\mapsto f(\xi,\cdot)|_X$ need not take values in a separable Banach space, and its measurability is therefore not automatic.
Instead, we reduce convergence $d_X(\Enu{f}, \Ef{f})\to 0$ to the Glivenko--Cantelli property of a class of difference quotients indexed by the countable dense set $Q$ from \Cref{as:definable}. The cover $\{\Xi_\ell\}_{\ell\in I}$ from \Cref{as:definable} allows flexibility on the modeling side. On each piece, the restrictions of the slices $\{f(\xi,\cdot)\mid \xi\in\Xi_\ell\}$ to $Q$ are uniformly definable in an NIP structure in a certain sense. 
Since elementary combinations of definable functions remain definable, the resulting collection of subgraphs of difference quotients is uniformly definable and therefore forms a VC-class.
The desired convergence of $d_X(\Enu{f}, \Ef{f})$ then follows by combining the corresponding estimates across different pieces.

When Assumptions~\ref{as:f-S} and \ref{as:definable} hold, under slightly stronger integrability assumptions and the requirement that $|I|<\infty$, we obtain the following convergence rate.

\begin{theorem}[rate of convergence]\label{thm:lip-rate}
Suppose that $(\Xi,\mathcal{A})$, $f:\Xi\times\reals^d\to\reals$, and $X\subset\reals^d$ with $0\in X$ satisfy Assumptions~\ref{as:f-S} and \ref{as:definable} with Lipschitz modulus $L$ and $|I| < \infty$. 
Let $\bm{\xi}, \bm{\xi}^1,\bm{\xi}^2,\ldots$ be $\Xi$-valued iid random variables on the complete probability space $(\Omega,\mathcal{F},\mathbb{P})$ with a law $P$.
If, in addition, $
\Ex[|f(\bm{\xi},0)|^2]<\infty$ and $\Ex[L^2]<\infty$, then
\[
d_X(\Enu{f},\Ef{f})=O_{\mathbb{P}}(\nu^{-1/2}).
\]
In other words, for any $\delta > 0$, there exist a constant $C_\delta < \infty$ and $\nu_0 \in \nats$ such that for all $\nu \geq \nu_0$, with probability at least $1-\delta$,  one has 
$d_X(\Enu{f},\Ef{f}) \leq C_\delta \nu^{-1/2}.$
\end{theorem}

\Cref{thm:lip-rate} may be viewed as a uniform tightness result for the sequence $\{\sqrt{\nu}d_X(\Enu{f},\Ef{f})\}_\nu$. Its proof is deferred to \Cref{sec:prf-thm:lip-rate}. 
The qualitative nature of \Cref{lem:nip-vc} prevents us from obtaining explicit bounds for the constants $C_\delta$ from the stated assumptions.
Nevertheless, the constants $C_\delta$ are related to the probability $\delta$, the dimension $d$, and the VC-dimensions of certain function classes. Explicit bounds may be obtainable in more structured settings, but we do not pursue this direction here. 
We emphasize that \Cref{as:definable} in \Cref{thm:lip-rate} cannot be replaced by \Cref{as:separable}, and that the finite-piece condition $|I|<\infty$ cannot be removed without sacrificing the stated rate of convergence.

\begin{proposition}[slow convergence]\label{prop:norate}
Let $\Xi=\nats$, $\mathcal{A}=2^\nats$, $d=1$, and $X=[0,1]$. There exists a function $f:\Xi\times\reals\to\reals$ such that, for every probability law $P$ on $\Xi$,  Assumptions~\ref{as:f-S}, \ref{as:separable}, and \ref{as:definable} hold with $|I|=\infty$,  $L(\xi)=1$, and $f(\xi,\cdot)$ piecewise affine for every $\xi\in\Xi$. Moreover, 
 for every sequence $r^\nu \in (0,1)\downarrow0$, there exists a law $P$ on $\Xi$ such that, for any iid sequence $\bm{\xi},\bm{\xi}^1,\bm{\xi}^2,\ldots$ with law $P$ and defined on a complete probability space $(\Omega,\mathcal{F},\mathbb{P})$, one has
\[
d_X(\Enu{f},\Ef{f})\neq O_{\mathbb{P}}(r^\nu).
\]
\end{proposition}

The proof of \Cref{prop:norate} constructs a sequence of piecewise affine functions that relates the Lipschitz pseudometric to the \emph{missing mass problem} studied in \cite{berend2012missing}, whose convergence can be arbitrarily slow; see \Cref{sec:prf:prop:norate} for the formal proof.
Since $L(\xi)\equiv 1$, the square-integrability condition in \Cref{thm:lip-rate} holds trivially. Nevertheless, the convergence in \Cref{thm:lln} can be arbitrarily slow. 
By contrast, for sample averages of $\reals^m$-valued random variables, finite second moments allow one to invoke a central limit theorem (CLT) and obtain the canonical rate $O_{\mathbb{P}}(\nu^{-1/2})$. This contrast highlights the fundamental difference between averaging random vectors in a Euclidean space and averaging random functions in a Lipschitz space.

\subsection{Proofs}\label{sec:prf:sec:lln}

Throughout the proofs, we work with appropriate probability-one events on which all relevant almost sure limits hold. Such events exist because only countably many almost sure statements are used and $\mathcal{F}$ is $\mathbb{P}$-complete; for notational simplicity, we do not name them explicitly. 
 
\subsubsection{Proof of \Cref{thm:lln} with Assumptions~\ref{as:f-S}~and~\ref{as:separable}}\label{sec:prf-countable}

By \Cref{as:f-S}, we have $f(\xi,\cdot)|_X\in\Lip(X)$ for $P$-a.e.~$\xi$.  By redefining $f$ on a $P$-null set according to Assumptions~\ref{as:f-S}~and~\ref{as:separable} if necessary, we may assume without loss of generality that $f(\xi,\cdot)\equiv 0$ for $\xi$ in that set.
Let $B$ be the $\|\cdot\|_{\lip}$-closure of the separable subspace in \Cref{as:separable}; hence $(B, \|\cdot\|_{\lip})$ is a separable Banach space.  Let $F:\Xi \to B$ be the mapping $\xi \mapsto f(\xi, \cdot)|_X\in B$. We view $F(\bm{\xi}):\Omega \to B$ as a $B$-valued random variable. To this end, it suffices to show that $F$ is
$\mathcal{A}/\mathcal{B}(B)$-measurable. 
Let $Q$ be a countable dense subset of $X$.
By restricting the defining supremum to the countable set $Q$, one sees that the function $
\xi\mapsto \glip(f(\xi,\cdot)|_X-b)$
is $\mathcal A$-measurable for any fixed and deterministic $b\in B \subset \Lip(X)$. Since
$\xi\mapsto |f(\xi,0)-b(0)|$ is also $\mathcal{A}$-measurable, the function
$\xi\mapsto \|f(\xi,\cdot)|_X-b\|_{\lip}$ is $\mathcal{A}$-measurable.
Let $\{b_k\}_{k \in \nats}$ be a countable dense subset of $B$. For each $k \in \nats$ and any $r \in \mathbb{Q}$, we have
\[
F^{-1}(\{b \in B \mid \| b - b_k \|_{\lip} < r\}) = \{\xi \in \Xi \mid  \| f(\xi, \cdot)|_X - b_k \|_{\lip} < r\} \in \mathcal{A},
\]
so that $F$ is
$\mathcal{A}/\mathcal{B}(B)$-measurable.
Since $B$ is separable, by \cite[Lemma 11.37]{aliprantis2006infinite},   $F(\bm{\xi}), F(\bm{\xi}^1), F(\bm{\xi}^2), \ldots$ are strongly measurable $B$-valued iid random variables on
$(\Omega,\mathcal F,\mathbb P)$. 
For any $\xi \in \Xi$, we have
\[
 \|F(\xi)\|_{\lip} = \max\{|f(\xi, 0)|, \glip(f(\xi, \cdot)|_X)\} \leq |f(\xi, 0)| + L(\xi).
\]
Therefore, $\Ex [\|F(\bm{\xi})\|_{\lip}] \leq \Ex[|f(\bm{\xi}, 0)|] + \Ex[L(\bm{\xi})] < \infty$, so that $F(\bm{\xi})$ is Bochner integrable by \cite[Theorem 11.44]{aliprantis2006infinite}; see also \cite[p.~42]{ledoux1991probability}.
Using \cite[Corollary 7.10]{ledoux1991probability}, one has 
\[
\Big\| \tfrac{1}{\nu} \nsum_{i=1}^\nu F(\bm{\xi}^i) - \Ex[F(\bm{\xi})]\Big\|_{\lip} \to 0,\qquad \mathbb{P}\text{-a.s.},
\]
where $\Ex[F(\bm{\xi})]\in B$ denotes the Bochner integral of $F(\bm{\xi})$.
Now, we relate the above convergence to that of $\Enu{f}$ and $\Ef{f}$. For any $x \in X$ and $\nu \in \nats$, we have
\[
\Big(\tfrac{1}{\nu} \nsum_{i=1}^\nu F(\bm{\xi}^i)\Big)(x) = \tfrac{1}{\nu} \nsum_{i=1}^\nu f(\bm{\xi}^i, x) = \Enu{f}(x).
\]
For each $x\in X$, define the evaluation map $\eval_x:B \to \reals$ by $\eval_x(b) = b(x)$ for $b \in B$. This map is linear. Meanwhile, for any $b \in B$, one has
\[
|\eval_x(b)|\leq |b(0)| + \glip(b)\|x\| \leq (1+\|x\|)\|b\|_{\lip},
\]
hence the linear map $\eval_x$ is bounded. From \cite[Lemma 11.45]{aliprantis2006infinite}, for any $x \in X$, 
\[
\Ex[F(\bm{\xi})](x) = \eval_x (\Ex[F(\bm{\xi})]) = 
\Ex[\eval_x (F(\bm{\xi}))]
= \Ex[f(\bm{\xi},x)] = \Ef{f}(x).
\]
Therefore, as $\nu \to \infty$, we obtain
\[
d_X(\Enu{f}, \Ef{f}) = \|(\Enu{f} - \Ef{f})|_X \|_{\lip} =  \Big\| \tfrac{1}{\nu} \nsum_{i=1}^\nu F(\bm{\xi}^i) - \Ex[F(\bm{\xi})]\Big\|_{\lip} \to 0,\qquad \mathbb{P}\text{-a.s.}
\]

\subsubsection{Proof of \Cref{thm:lln} with Assumptions~\ref{as:f-S}~and~\ref{as:definable}}\label{sec:prf-def}
 Let $\{\Xi_\ell\}_{\ell\in I}$ be the pieces in \Cref{as:definable}. After reindexing and setting $\Xi_\ell=\emptyset$ for unused indices, we may assume without loss of generality that $I=\nats$. 
 After removing a $P$-null set from the pieces if necessary, we may also assume that $f(\xi,\cdot)$ is $L(\xi)$-Lipschitz on $X$ for every $\xi\in\cup_\ell\Xi_\ell$.
 Let $Q = \{q_k\}_{k \in \nats}$ be a countable dense subset of $X$ from \Cref{as:definable}.
If $X$ is a singleton, then necessarily $X=\{0\}$ and $d_X(\Enu{f},\Ef{f})=|\Enu{f}(0)-\Ef{f}(0)|$ and \Cref{thm:lln} reduces to the classical SLLN. Hence, in what follows, we assume that $X$ contains at least two distinct points.
For every $x\in X$,  one has $
|\Ef{f}(x)| \leq 
\Ex[|f(\bm{\xi},x)|]
\le \Ex[|f(\bm{\xi},0)|]+(\Ex L)\|x\|.$
 Since $\Ex[|f(\bm{\xi},0)|] < \infty$ and $\Ex L < \infty$, the function $\Ef{f}:\reals^d \to \Reals$ is real-valued on $X$.
For any $x,y\in X,x\neq y$, define $g_{x,y}:\Xi \to \reals$ by
\[
g_{x,y}(\xi) = \frac{f(\xi, x) - f(\xi, y)}{\|x-y\|}.
\]
For any $\mathcal{A}$-measurable set $S \subset \Xi$, define $
\mathcal{G}_S= \{g_{x,y}\bm{1}_S \mid x,y \in Q,x\neq y\}$. Observe the  identity
\[
\glip_X (\Enu{f}-\Ef{f})  = 
\sup_{x,y\in X, x \neq y} \frac{|(\Enu{f}(x)-\Ef{f}(x))- (\Enu{f}(y)- \Ef{f}(y))|}{\|x - y\|}
= \sup_{g \in \mathcal{G}_\Xi} |\Ex^\nu g -  \Ex g|. 
\]
Thus, the convergence of $\glip_X (\Enu{f}-\Ef{f})$ is governed by the empirical process indexed by $\mathcal{G}_{\Xi}$.
To see the last equality, fix $x,y \in X, x \neq y$. Let $x_k, y_k \in Q$ be such that $(x_k, y_k) \to (x, y)$. Since $x\neq y$, for all sufficiently large $k \in \nats$, one has $x_k \neq y_k$. Then, for $P$-a.e.~$\xi \in \Xi$, one has $f(\xi, x_k) \to f(\xi, x)$, $f(\xi, y_k) \to f(\xi, y)$, and $\|x_k - y_k\| \to \|x - y\|>0$, so that we get $g_{x_k,y_k}(\xi) \to g_{x,y}(\xi)$ for $P$-a.e.~$\xi$, and the claim follows by the dominated convergence theorem.

By \Cref{as:definable}, each $\Xi_\ell$ is measurable with respect to $\mathcal{A}^P$.
 For any $\ell \in \nats$, there exist $C_\ell \subset \Xi_\ell \subset C_\ell'$ such that $C_\ell, C_\ell' \in \mathcal{A}$ and $P(C_\ell' \setminus C_\ell) = 0$; see \cite[Section 1.5]{cohn2013measure}. Let $B_\ell=C_\ell\setminus(\cup_{k < \ell} C_k)$ and $A_\ell = \cup_{k \in [\ell]} C_k = \sqcup_{k \in [\ell]} B_k$, so that $\{A_\ell\}_\ell$ and $\{B_\ell\}_\ell$ are both $\mathcal{A}$-measurable. One has $P(\cup_\ell C_\ell)=1$, since
\[
1=P(\cup_\ell \Xi_\ell)\leq P(\cup_\ell C_\ell') \leq P(\cup_\ell C_\ell) + P(\cup_\ell (C_\ell'\setminus C_\ell))=P(\cup_\ell C_\ell).
\]
Therefore, $P(A_\ell) \uparrow 1.$  
A key ingredient of our proof is the following VC-subgraph lemma.
\begin{lemma}\label{lem:Gm-VC}
	For each $\ell \in \nats$, the family $\mathcal{G}_{B_\ell}$ is a VC-subgraph class.
\end{lemma}
\begin{proof}
	For any $x,y \in \reals^d$, let
\[
S_{x,y}^\ell = \{(\xi, t) \in B_\ell \times \reals \mid t\|x - y\| < f(\xi, x) - f(\xi, y)\}, 
\]
which is empty when $x=y$.
 Define the following formula in the language of $\mathcal{M}_\ell$: 
\[
\eta_\ell(z,t;u,v,a)
=
(\exists r_1,r_2)\quad 
\phi_\ell(u,r_1;z)\wedge \phi_\ell(v,r_2;z)
\wedge t a<r_1-r_2
\]
with $r_1,r_2 \in \reals, z \in \reals^{k_\ell}, t \in \reals, u \in \reals^d, v \in \reals^d, a \in \reals$. Since $\mathcal{M}_\ell$ has NIP, every formula in its language has NIP, so that $\eta_\ell$ has NIP.
	Using \Cref{lem:nip-vc}, we know that the family
	\[
	\Big\{\{(z,t) \in \reals^{k_\ell+1} \mid \mathcal{M}_\ell \models \eta_\ell(z, t; x, y, a)\} \Bigm| x,y \in \reals^d, a \in \reals\Big\} 
	\]
	is a VC-class. Hence, its subfamily
	\[
	\Big\{Z_{x,y}^\ell=\{(z,t) \in \reals^{k_\ell+1} \mid \mathcal{M}_\ell \models \eta_\ell(z, t; x, y, \|x-y\|)\} \Bigm| x,y \in Q, x\neq y\Big\}
	\]
	is also a VC-class. 
	For each $\ell \in \nats$, consider the following mapping
	\[
		\Psi_\ell:B_\ell \times \reals \longrightarrow \reals^{k_\ell} \times \reals,\qquad
		(\xi, t)  \longmapsto (z_{\ell,\xi}, t),
	\]
	where $z_{\ell,\xi} \in \reals^{k_\ell}$ is the parameter in \Cref{as:definable} such that \eqref{eq:def-q}  holds. 
	For any $x,y \in Q,x\neq y$, by \Cref{as:definable} and the inclusions $B_\ell\subset  C_\ell\subset \Xi_\ell$, we observe
	\[
	\begin{aligned}
	S_{x,y}^\ell &= \{(\xi, t) \in B_\ell\times \reals \mid (\exists r_1,r_2,a)\  r_1=f(\xi,x),r_2=f(\xi,y),a=\|x - y\|, t a < r_1 - r_2 \}  \\
	&= \{(\xi, t) \in B_\ell\times \reals \mid (\exists r_1,r_2,a)\ \mathcal{M}_\ell\models \phi_\ell(x,r_1; z_{\ell,\xi})\wedge \phi_\ell(y,r_2; z_{\ell,\xi}),a=\|x - y\|, t a < r_1 - r_2 \} \\
	&= \{(\xi, t) \in B_\ell\times \reals \mid \mathcal{M}_\ell \models \eta_\ell(z_{\ell,\xi}, t; x,y, \|x-y\|)\} = \Psi^{-1}_\ell(Z^\ell_{x,y}).
	\end{aligned}
	\]
	By \Cref{lem:vc-rules}, the family $\{S^\ell_{x,y} \mid x,y \in Q, x \neq y\}$ is a VC-class. 
	Note that
		\[
\Big\{\{(\xi, t) \mid t < g_{x,y}(\xi)\bm{1}_{B_\ell}(\xi)\} \Bigm| x, y \in Q, x \neq y \Big\} = \Big\{S^\ell_{x,y} \cup \big(B_\ell^\compl \times (-\infty, 0)\big) \Bigm| x,y \in Q, x \neq y\Big\}.
\]
Since a class containing only one set is trivially a VC-class, by \Cref{lem:vc-rules}, the above collection  is a VC-class. Therefore, the collection $\mathcal{G}_{B_\ell}$ is a VC-subgraph class. 
\end{proof}

We proceed by combining the analysis over the pieces $\{\Xi_\ell\}_{\ell\in I}$, adapting an argument from \cite[Theorem~4]{van2000preservation}.
For any distinct $x,y \in Q$, the function $g_{x,y}$ is $\mathcal{A}$-measurable, since $f(\cdot,z)$ is $\mathcal{A}$-measurable for every $z \in \reals^d$. For any $\mathcal{A}$-measurable set $S$, the function $g_{x,y}\mathbf{1}_S$ is $\mathcal{A}$-measurable as well.
Thus, by countability of $Q$, the families $\mathcal{G}_{A_\ell}$ and $\mathcal{G}_{B_\ell}$ are both pointwise measurable \cite[Example 2.3.4]{vanderVaartWellner.23}, hence $P$-measurable in the sense of \cite[Definition 2.3.3]{vanderVaartWellner.23}.
Since $f(\xi, \cdot)$ is $L(\xi)$-Lipschitz on $X$ for $P$-a.e.~$\xi$, one has $|g_{x,y}(\xi)\bm{1}_S(\xi)| \leq |g_{x,y}(\xi)| \leq L(\xi)$ for any $x ,y \in X,x\neq y$. By $\Ex L < \infty$, $\mathcal{G}_{B_\ell}$ has an integrable envelope $L$.  From \Cref{lem:Gm-VC}, the family $\mathcal{G}_{B_\ell}$ is a VC-subgraph class. 
	By the control of uniform entropy from \cite[Theorem 2.6.7]{vanderVaartWellner.23} and then invoking \cite[Theorem 2.8.1]{vanderVaartWellner.23}, the class $\mathcal{G}_{B_\ell}$ is $P$-Glivenko--Cantelli. Hence, for each $k \in \nats$,
	\[
	\lim_{\nu \to \infty} \sup_{g \in \mathcal{G}_{B_k}} |\Ex^\nu g -  \Ex g| = 0,\qquad \mathbb{P}\text{-a.s.}
	\]
	Since $\bm{1}_{A_\ell} = \sum_{k=1}^\ell \bm{1}_{B_k}$, for any $\ell \in \nats$, $\mathbb{P}$-a.s., one has
\[
\limsup_{\nu \to \infty} \sup_{g \in \mathcal{G}_{A_\ell}} |\Ex^\nu g -  \Ex g| \leq \nsum_{k=1}^\ell \limsup_{\nu \to \infty} \sup_{g \in \mathcal{G}_{B_k}} |\Ex^\nu g -  \Ex g| = 0.
\] 
For any $\ell \in \nats$ and $g \in \mathcal{G}_\Xi$,  observe that
\[
\begin{aligned}
|\Ex^\nu g -  \Ex g|
&\leq |\Ex^\nu [g\bm{1}_{A_\ell}] -  \Ex [g\bm{1}_{A_\ell}]| + |\Ex^\nu [g\bm{1}_{A_\ell^\compl}]-\Ex [g\bm{1}_{A_\ell^\compl}]| \\
&\leq |\Ex^\nu [g\bm{1}_{A_\ell}] -  \Ex [g\bm{1}_{A_\ell}]| + \Ex^\nu [L\bm{1}_{A_\ell^\compl}]+\Ex [L\bm{1}_{A_\ell^\compl}].
\end{aligned}
\]
Therefore, for any $\ell \in \nats$, \[
\glip_X (\Enu{f}-\Ef{f})  = \sup_{g \in \mathcal{G}_\Xi} |\Ex^\nu g -  \Ex g| 
\leq \sup_{g \in \mathcal{G}_{A_\ell}} |\Ex^\nu g -  \Ex g| + \Ex^\nu [L\bm{1}_{A_\ell^\compl}]+\Ex [L\bm{1}_{A_\ell^\compl}].
\]
Using $\Ex [L]<\infty$, by SLLN, we have $\Ex^\nu [L\bm{1}_{A_\ell^\compl}] \to \Ex [L\bm{1}_{A_\ell^\compl}]$, $\mathbb{P}$-a.s., as $\nu \to \infty$. 
Since $P(A_\ell^\compl)\downarrow 0$ and $\Ex [L]<\infty$, by the dominated convergence theorem, we have $\Ex [L\bm{1}_{A_\ell^\compl}]\to 0$.
Then, $\mathbb{P}$-a.s., we have
\[
\limsup_{\nu \to \infty}\glip_X (\Enu{f}-\Ef{f})  \leq \limsup_{\ell \to \infty} 2\Ex [L\bm{1}_{A_\ell^\compl}] = 0.
\]
Since $\Ex[|f(\bm{\xi}, 0)|] < \infty$, by SLLN, one has  $\Enu{f}(0) \to \Ef{f}(0)$, $\mathbb{P}$-a.s. Therefore, we have
\[
d_X(\Enu{f}, \Ef{f}) = \max\{|\Enu{f}(0)- \Ef{f}(0)|, \glip_X(\Enu{f}-\Ef{f})\} \to 0, \qquad \mathbb{P}\text{-a.s.}
\]

\subsubsection{Proof of \Cref{thm:lip-rate}}\label{sec:prf-thm:lip-rate}

Since $|I|<\infty$, after reindexing we may assume that $I=\nats$ and that $\Xi_\ell=\emptyset$ for all $\ell>T$, for some $T\in\nats$.
Let $\{A_\ell\}_{\ell \in \nats}$ and $\{B_\ell\}_{\ell \in \nats}$ be as defined in the proof of \Cref{thm:lln}. Then $P(A_{T}) = 1$.
If $X$ is a singleton, then necessarily $X=\{0\}$, and the claim reduces to a CLT. Hence, we assume henceforth that $X$ contains at least two distinct points.
Let $\mathcal{H}=\{f(\cdot, 0)\}\cup \mathcal{G}_{A_{T}}$.
Observe that $L$ is an integrable envelope of $\mathcal{G}_{A_{T}}$. Let $L_2(\xi)=|f(\xi, 0)|+L(\xi)$, so that $L_2$ is $\mathcal{A}$-measurable. By assumption, $
\Ex [L_2^2] \leq 2\Ex [|f(\bm{\xi}, 0)|^2]+2\Ex [L^2] <\infty$; hence $L_2$
is a square-integrable envelope of $\mathcal{H}$. As in the proof of \Cref{thm:lln}, the class $\mathcal{G}_{A_T}$ and also $\mathcal{H}$ are pointwise measurable.
Moreover, for every $\delta>0$, the classes
\[
\mathcal H_\delta
=
\{h_1-h_2 \mid h_1,h_2\in\mathcal{H},\ 
  \textstyle{\int} (h_1-h_2)^2 \dd P<\delta^2\},\quad 
  \mathcal{H}_\infty^2
=
\{(h_1-h_2)^2 \mid h_1,h_2\in\mathcal{H}\}
\]
are pointwise measurable, by the dominated convergence theorem and  the square-integrable envelope $L_2$. Since $A_T = \sqcup_{\ell \in [T]}B_\ell$, the collection of subgraphs of functions in  $\mathcal{H}$ is contained in 
\[
 \mathcal{D}\cup \Big\{C_1 \cup \cdots \cup C_{T} \cup \big(A_{T}^\compl \times (-\infty, 0)\big) \Bigm| C_\ell \in \mathcal{C}_\ell, \ell \in [T]\Big\},
 \] 
 where
\[
\mathcal{D}=\Big\{\{(\xi, r) \mid r < f(\xi, 0)\}\Big\}, \qquad \mathcal{C}_\ell=\Big\{ \{(\xi, r) \in B_\ell \times \reals \mid r < g_{x,y}(\xi) \} \Bigm| x, y \in Q, x \neq y\Big\}.
\]
For each $\ell\in[T]$, the class $\mathcal{C}_\ell$ is a VC-class by the proof of \Cref{lem:Gm-VC}, while $\mathcal{D}$
is a singleton class and hence is trivially a VC-class.
By \Cref{lem:vc-rules}, the class of the above finite unions is
a VC-class. Therefore, $\mathcal{H}$ is a VC-subgraph class.
  By the control of uniform entropy from \cite[Theorem 2.6.7]{vanderVaartWellner.23} and then invoking \cite[Theorem 2.5.2]{vanderVaartWellner.23}, the class $\mathcal{H}$ is $P$-Donsker. Thus the empirical process
\[
\sqrt{\nu}(\Ex^\nu h-\Ex h),
\qquad h\in\mathcal{H},
\]
converges weakly in $\ell^\infty(\mathcal{H})$ to a centered tight Gaussian process $\mathbb{G}_P$.
Since $P(A_{T}) = 1$, we have $\glip_X(\Enu{f}-\Ef{f})=\sup_{g \in \mathcal{G}_\Xi} |\Ex^\nu [g\bm{1}_{A_{T}}]-\Ex [g\bm{1}_{A_{T}}]|$, $\mathbb{P}$-a.s. Hence, one has
\[
\sqrt{\nu}\,d_X(\Enu{f},\Ef{f})
=\nsup_{h\in\mathcal{H}}|\sqrt{\nu}(\Ex^\nu h-\Ex h)|,\qquad \mathbb{P}\text{-a.s.}
\]
The mapping $
z\mapsto \sup_{h\in\mathcal{H}}|z(h)|$
is the supremum norm on $\ell^\infty(\mathcal{H})$, hence is continuous. The continuous mapping theorem yields
\[
\sqrt{\nu}\,d_X(\Enu{f},\Ef{f})
\to
\nsup_{h\in\mathcal{H}}|\mathbb{G}_P h|\qquad\text{in distribution.}
\]
The limit is finite almost surely because $\mathbb G_P$ is tight as an $\ell^\infty(\mathcal H)$-valued random element. Hence $
\sqrt{\nu}\,d_X(\Enu{f},\Ef{f})=O_{\mathbb P}(1).$

\subsubsection{Proof of \Cref{prop:norate}}\label{sec:prf:prop:norate}
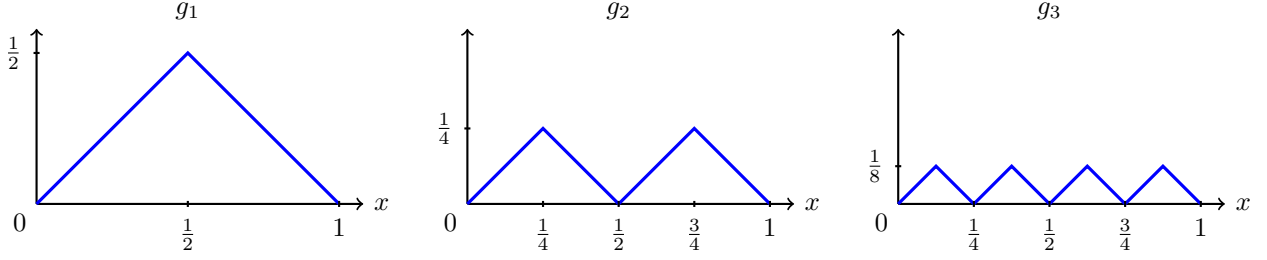
\begin{figure}[t]
\centering

\begin{minipage}{0.31\textwidth}
\centering
\begin{tikzpicture}[
    x=4cm,
    y=4cm,
    font=\small,
    axis/.style={->, line width=0.8pt},
    tick/.style={line width=0.8pt}
]

\draw[axis] (0,0) -- (1.08,0) node[right] {$x$};
\draw[axis] (0,0) -- (0,0.58);

\draw[blue, very thick]
    (0,0)
    -- (0.5,0.5)
    -- (1,0);

\draw[tick] (0.5,0.01) -- (0.5,-0.01);
\node[below] at (0.5,-0.015) {$\frac12$};

\draw[tick] (1,0.01) -- (1,-0.01);
\node[below] at (1,-0.015) {$1$};

\draw[tick] (0.01,0.5) -- (-0.01,0.5);
\node[left] at (-0.015,0.5) {$\frac12$};

\node[below left] at (0,0) {$0$};

\node at (0.5,0.64) {$g_1$};

\end{tikzpicture}
\end{minipage}
\hfill
%
\begin{minipage}{0.31\textwidth}
\centering
\begin{tikzpicture}[
    x=4cm,
    y=4cm,
    font=\small,
    axis/.style={->, line width=0.8pt},
    tick/.style={line width=0.8pt}
]

\draw[axis] (0,0) -- (1.08,0) node[right] {$x$};
\draw[axis] (0,0) -- (0,0.58);

\draw[blue, very thick]
    (0,0)
    -- (0.25,0.25)
    -- (0.5,0)
    -- (0.75,0.25)
    -- (1,0);

\foreach \x/\lab in {
    0.25/{\frac14},
    0.5/{\frac12},
    0.75/{\frac34},
    1/{1}
}{
    \draw[tick] (\x,0.01) -- (\x,-0.01);
    \node[below] at (\x,-0.015) {$\lab$};
}

\draw[tick] (0.01,0.25) -- (-0.01,0.25);
\node[left] at (-0.015,0.25) {$\frac14$};

\node[below left] at (0,0) {$0$};

\node at (0.5,0.64) {$g_2$};

\end{tikzpicture}
\end{minipage}
\hfill
%
\begin{minipage}{0.31\textwidth}
\centering
\begin{tikzpicture}[
    x=4cm,
    y=4cm,
    font=\small,
    axis/.style={->, line width=0.8pt},
    tick/.style={line width=0.8pt}
]

\draw[axis] (0,0) -- (1.08,0) node[right] {$x$};
\draw[axis] (0,0) -- (0,0.58);

\draw[blue, very thick]
    (0,0)
    -- (0.125,0.125)
    -- (0.25,0)
    -- (0.375,0.125)
    -- (0.5,0)
    -- (0.625,0.125)
    -- (0.75,0)
    -- (0.875,0.125)
    -- (1,0);

\foreach \x/\lab in {
    0.25/{\frac14},
    0.5/{\frac12},
    0.75/{\frac34},
    1/{1}
}{
    \draw[tick] (\x,0.01) -- (\x,-0.01);
    \node[below] at (\x,-0.015) {$\lab$};
}

\draw[tick] (0.01,0.125) -- (-0.01,0.125);
\node[left] at (-0.015,0.125) {$\frac18$};

\node[below left] at (0,0) {$0$};

\node at (0.5,0.64) {$g_3$};

\end{tikzpicture}
\end{minipage}

\caption{The triangular-wave functions $g_k$ for $k=1,2,3$.}
\label{fig:triangular-wave}
\end{figure}

Let $X=[0,1], d= 1, \Xi = \nats, \mathcal{A}=2^\nats$, and fix any $r^\nu \in (0,1) \downarrow 0$.
For any $k \in \nats$, let $g_k: \reals \to \reals$ be the triangular-wave function (see \Cref{fig:triangular-wave}):
\[
g_k(x)=\int_0^x \bm{1}_{[0,1]}(t)(-1)^{\lfloor 2^kt\rfloor} \dd t.
\]
It is easy to verify that $g_k(0) = 0$, $\glip_X g_k = 1$, and $g_k$ is piecewise affine. 
For any $K \in \nats$ and any $(b_1, b_2, \ldots) \in \reals^\nats$, there exists a point $x_K  \in [0,1]$ such that $g_k$ is locally affine at $x_K$ with $g_k'(x_K)=\sgn(b_k)$ for each $k \leq K$ such that $b_k \neq 0$. Hence, we obtain
\[
\nsum_{k=1}^K |b_k|=\nsum_{k=1}^K |b_k| \|g_k|_X\|_{\lip} \geq \Big\|\nsum_{k=1}^K b_k g_k|_X \Big\|_{\lip} 
\geq \lip \Big(\nsum_{k=1}^K b_k g_k|_X\Big)(x_K) = 
 \nsum_{k=1}^K |b_k|.
 \]
Hence $\|\nsum_{k=1}^K b_k g_k|_X \|_{\lip} = \nsum_{k=1}^K |b_k|$ for any $K \in \nats$. 
When $\sum_{k=1}^\infty |b_k| < \infty$, for any $K_1, K_2 \in \nats$, we have $\|\nsum_{k=K_1}^{K_2} b_k g_k |_X\|_{\lip} \leq \sum_{k=K_1}^{K_2} |b_k|$, so that the sequence $(\sum_{k=1}^K b_kg_k|_X)_K$ is Cauchy.
Since $\Lip(X)$ is complete, we obtain $\sum_{k=1}^\infty b_kg_k|_X \in \Lip(X)$ with 
\[
\Big\|\nsum_{k=1}^\infty b_k g_k|_X \Big\|_{\lip} 
=
\lim_{K \to \infty}  \Big\|\nsum_{k=1}^K b_k g_k|_X \Big\|_{\lip}
= 
 \nsum_{k=1}^\infty |b_k|.
\]

Let $\delta_k$ be the Dirac measure at $k$, $p_k \in [0,1]$ such that $\sum_{k=1}^\infty p_k = 1$, and $P = \sum_{k=1}^\infty p_k\delta_k$.
Let  $\bm{\xi}, \bm{\xi}^1, \bm{\xi}^2,\ldots$ be iid $\nats$-valued discrete random variables on $(\Omega,\mathcal{F},\mathbb{P})$ with the law $P$. For any $\nu \in \nats$, define the empirical estimate $\bm{p}^\nu_k = \tfrac{1}{\nu}\sum_{i=1}^\nu \bm{1}_{\{\bm{\xi}^i = k\}}$. Define $f:\nats \times \reals \to \reals$ by $f(\xi, x) = g_\xi (x)$. Then, Assumptions~\ref{as:f-S},~\ref{as:separable},~and~\ref{as:definable} hold.
Note that $\nsum_{k=1} ^\infty |\bm{p}^\nu_k - p_k|\leq 2$, so that
\[
d_X(\Enu{f}, \Ef{f}) = \Big\|\nsum_{k=1}^\infty (\bm{p}^\nu_k - p_k) g_k|_X \Big\|_{\lip} = \nsum_{k=1} ^\infty |\bm{p}^\nu_k - p_k| \geq \nsum_{k : \bm{p}^\nu_k = 0} |p_k|,
\]
where the right-hand side is the \emph{missing mass} as defined in \cite[(1)]{berend2012missing}. Let $\bm{U}^\nu=\nsum_{k : \bm{p}^\nu_k = 0} |p_k|$. Then, by \cite[(2) and Proposition 4]{berend2012missing}, there exists a specific law  $P$ on $\Xi$ such that 
\[
\Ex[\bm{U}^\nu] > (\sqrt{r^\nu} + \nu^{-1/4})/2,\qquad 
\mathbb{P}(|\bm{U}^\nu  - \Ex[\bm{U}^\nu]| \geq \epsilon) \leq 2\exp(-\nu\epsilon^2).
\] 
Let $\epsilon = (\sqrt{r^\nu} + \nu^{-1/4})/4$ and we compute
\[
\mathbb{P}\Big(\bm{U}^\nu \leq \tfrac{\sqrt{r^\nu} + \nu^{-1/4}}{4}\Big) \leq \mathbb{P}\Big(|\bm{U}^\nu  - \Ex[\bm{U}^\nu]| \geq \tfrac{\sqrt{r^\nu} + \nu^{-1/4}}{4}\Big) \leq 2\exp(-\sqrt{\nu}/16) \to 0,
\]
as $\nu \to \infty$.
For any $M > 0$, one has $\sqrt{r^\nu} + \nu^{-1/4} > 4Mr^\nu$ eventually. Hence 
\[
\limsup_{\nu \to \infty} \mathbb{P}\Big(d_X(\Enu{f}, \Ef{f}) \leq  M r^\nu \Big) \leq \limsup_{\nu \to \infty}
 \mathbb{P}\Big(\bm{U}^\nu \leq  \tfrac{ \sqrt{r^\nu} + \nu^{-1/4} }{4}\Big) = 0,
\]
so that $d_X(\Enu{f}, \Ef{f}) \neq O_{\mathbb{P}}(r^\nu)$.

\section{Establishing Separability and Definability}\label{sec:verify}

In this section, we examine the abstract conditions in Assumptions~\ref{as:separable} and~\ref{as:definable} more closely and highlight notable function classes that satisfy either one of them. %

\subsection{Separability in \Cref{as:separable}}\label{sec:sep}

Recall that \Cref{as:separable} requires the family
$\{f(\xi,\cdot)|_X\mid \xi\in\Xi'\}$, for some $\Xi'$ with
$P(\Xi')=1$, to be contained in a separable subspace of
$(\operatorname{Lip}(X),\|\cdot\|_{\rm lip})$.
The following countability condition is an immediate sufficient condition.

\begin{proposition}[countability]\label{prop:count}
If the collection $E=\{f(\xi, \cdot)|_X \mid \xi \in \Xi\}$ is a countable subset of $\Lip(X)$, then	\Cref{as:separable} holds.
\end{proposition}

\begin{proof}
Since $E$ is countable, the set of all finite rational linear combinations of elements of $E$ is countable and dense in $\operatorname{span}E$. Hence, $\operatorname{span}E\subset\Lip(X)$ is separable.
\end{proof}

When $\bm{\xi}$ has a discrete distribution, we may restrict $\Xi$ to a countable set of full $P$-measure. Then $\{f(\xi,\cdot)|_X\mid \xi\in\Xi\}$ is countable, and \Cref{as:separable} follows from \Cref{prop:count}. Discrete distributions arise naturally in stochastic programming, for example, in scenario formulations; see \cite[Section 2.4.1]{shapiro2021lectures} and \cite[Chapter 3]{primer}. Countability also arises through the operation of \emph{quantization}, which plays a fundamental role in applications such as signal processing \cite{gray2002quantization}.

\begin{example}[quantization]
Let $J\subset\Xi$ be countable, and let $Q:\Xi\to J$ be called a \emph{quantizer}. Standard examples include rounding, truncation, and other discretizations arising from finite-precision arithmetic. If $f(Q(\xi),\cdot)|_X \in \Lip(X)$ for any $\xi \in \Xi$, then \Cref{as:separable} holds for $\{x \mapsto f(Q(\xi),x) \mid \xi \in \Xi\}$.
\end{example}

To further appreciate the role of the countability condition, it is instructive to see how a simple modification of the construction in \Cref{prop:counterex} mitigates its pathological behavior.

\begin{example}[remedy I]\label{ex:mod1}
Consider the setting of \Cref{prop:counterex}, but replace the uniform distribution of $\bm{\xi}$ by any distribution supported on $\mathbb{Q}\cap [0,1]$. Then, by \Cref{prop:count} and \Cref{thm:lln}, we have $d_X(\Enu{f}, \Ef{f}) \to 0$ for the function $f:\Xi \times \reals \to \reals$ in \Cref{prop:counterex}.
\end{example}
Although perhaps evident, we remark that countability of $X$ cannot replace countability of $\Xi$, since \Cref{prop:counterex} continues to hold when $X=\mathbb{Q}\cap [0,1]$.  While countability is convenient and easy to verify, separability in \Cref{as:separable} is strictly more general. We record the following result, whose proof is by definition, hence omitted.
\begin{proposition}[stochastic linear combination]\label{prop:sep-lin}
Let $h_k:\Xi\to\reals$ and $g_k:\reals^d\to\reals$ be such that $g_k|_X\in\Lip(X)$ for every $k\in\nats$. If, for $P$-a.e.~$\xi\in\Xi$,
\begin{equation}\label{eq:sep-lin}
\lim_{m\to\infty}
\Big\|
f(\xi,\cdot)|_X-\nsum_{k=1}^{m} h_k(\xi)g_k|_X
\Big\|_{\lip}
=0,
\end{equation}
then \Cref{as:separable} holds.
\end{proposition}

 A sufficient condition for the validity of \eqref{eq:sep-lin} is
\begin{equation}\label{eq:sep-lin-succ-cond}
f(\xi, x)= \nsum_{k=1}^{\infty} h_k(\xi) g_k(x), \qquad \text{and} \qquad 
 \nsum_{k=1}^{\infty} |h_k(\xi)|\, \|g_k|_X \|_{\lip} < \infty.	
\end{equation}
Note that the construction of $f$ in \Cref{prop:counterex} admits a similar representation to \eqref{eq:sep-lin-succ-cond}; cf.~\eqref{eq:ts25} with $h_k = \bit_k$. It is again instructive to consider modifications that restore the desired convergence without requiring $\Xi$ to be countable.

\begin{example}[remedy II]\label{ex:mod2}
For any $k \in \nats$, let $g_k:\reals \to \reals$ be as defined in \eqref{eq:ts25}. 
Let $f(\xi,x)=\nsum_{k=1}^{\infty} h_k(\xi)g_k(x)$ for some $h_k:\Xi \to \reals$. When $h_k(\xi)=\bit_k(\xi)$ for all $k \in \nats$, as in \eqref{eq:ts25}, the partial sums $\nsum_{k=1}^{m} h_k(\xi)g_k$ converge pointwise to $f(\xi,\cdot)$ as $m \to \infty$ for $P$-a.s.~$\xi \in \Xi$, but their restrictions to $X$ do not converge in the Lipschitz norm as required by \eqref{eq:sep-lin}. Hence, as shown in \Cref{prop:counterex}, the convergence $d_X(\Enu{f},\Ef{f})\to 0$ need not hold.
For a measurable $K:\Xi \to \nats$, under either of the following choices,
\[
h_k(\xi)=\bm{1}_{k \leq K(\xi)}\bit_k(\xi),\ \forall k,
\qquad \text{or} \qquad
h_k(\xi)=k^{-1.1}\bit_k(\xi),\ \forall k,
\]
the absolute summability condition in \eqref{eq:sep-lin-succ-cond} and \Cref{as:f-S} are satisfied. Hence, we recover the convergence $d_X(\Enu{f},\Ef{f})\to 0$ by \Cref{prop:sep-lin} and \Cref{thm:lln}. 
\end{example}

From \Cref{ex:mod1,ex:mod2}, one can recover the convergence $d_X(\Enu{f},\Ef{f})\to 0$ in the setting of \Cref{prop:counterex} either by modifying only the distribution of $\bm{\xi}$ while keeping $f$ fixed, or by modifying only $f$ while keeping the distribution of $\bm{\xi}$ fixed.

\subsection{Definability in \Cref{as:definable}}\label{sec:def}

We now turn to the definability condition. 
\subsubsection{Jointly definable functions}\label{sec:j-def}

We first consider functions jointly definable in an NIP structure. Let $\mathcal{M}$ be an NIP expansion of the real ordered field $\R-alg$. A function $f:\Xi\times\reals^d\to\reals$ is jointly definable in $\mathcal{M}$ if $\Xi\subset\reals^s$ for some $s\in\nats$ and $\gph f$ is definable in $\mathcal{M}$. As shown below, every jointly definable function satisfies \Cref{as:definable}.

\begin{proposition}[joint definability]\label{prop:j-def}
Given $\Xi \subset \reals^s$, if a function $f:\Xi \times \reals^d \to \reals$  is jointly definable in some NIP expansion $\mathcal{M}$ of $\R-alg$, then \Cref{as:definable} holds.
\end{proposition}
\begin{proof}
Since $f$ is definable in $\mathcal M$, the set $\gph f$ is definable by some formula
$\phi(\xi,x,r; z)$ with $\xi\in\reals^s$, $x\in\reals^d$, $r\in\reals$, and parameter $z \in \reals^k$. Let $I=\{1\},\Xi_\ell = \Xi, \mathcal{M}_\ell = \mathcal{M}, \phi_\ell(x,r;(\xi,z)) = \phi(\xi,x,r; z), k_\ell = s+k$ for $\ell \in I$.
Then \Cref{as:definable} holds for any countable dense $Q\subset X$.
\end{proof}

Functions jointly definable in an expansion of $\R-alg$ enjoy robust closure properties. In particular, joint definability is preserved under restriction to definable sets, finite Cartesian products, coordinate projections, composition with definable maps, pointwise sums and products, and finite pointwise maxima and minima; see \cite[Chapter~1,~\S2]{van1998tame}. These elementary closure properties do not require NIP. The NIP assumption is nevertheless mild and, by \Cref{prop:om-NIP}, is satisfied in particular by every o-minimal structure.

\begin{corollary}[joint o-minimality]\label{coro:j-omin}
Let $(\Xi, \mathcal{A})$ be a measurable space, and let $\bm{\xi},\bm{\xi}^1,\bm{\xi}^2,\ldots$ be $\Xi$-valued iid random variables on a complete probability space $(\Omega,\mathcal{F},\mathbb{P})$ with a law $P$. Suppose that $f:\Xi\times\reals^d\to\reals$ and $X\subset\reals^d$ satisfy \Cref{as:f-S}. If $f$ is jointly definable in an o-minimal expansion of $\R-alg$, then
\[
    \lim_{\nu\to\infty} d_X(\Enu{f},\Ef{f})=0,
    \qquad \mathbb{P}\text{-a.s.}
\]
If, in addition, $
\Ex[|f(\bm{\xi},0)|^2]<\infty$ and $\Ex[L^2]<\infty$, then for any $\delta > 0$, there exists a constant $C_\delta < \infty$ such that for all sufficiently large $\nu \in \nats$, with probability at least $1-\delta$, one has 
\[
 d_X(\Enu{f},\Ef{f}) \leq  C_\delta\nu^{-1/2}.
 \]
\end{corollary}
\begin{proof}
	The claim follows from \Cref{prop:om-NIP,prop:j-def}, \Cref{thm:lln,thm:lip-rate}.
\end{proof}

As illustrated in \Cref{sec:pre-model}, semialgebraic functions are definable in the structure $\R-alg$, which is itself o-minimal \cite[Corollary 2.2.11]{van1998tame}. Many important applications involve semialgebraic objective functions, including covariance estimation, sparse dictionary learning, phase retrieval, blind deconvolution, and matrix sensing; see the discussions in \cite[Section 6]{ruan2024subgradient} and \cite[Table 1]{davis2022graphical} for details. Moreover, the nonseparable $f$ in \Cref{ex:nsep} is obviously semialgebraic.

An important o-minimal expansion of $\R-alg$ is the real exponential field $\mathcal{R}_{\rm exp} = (\R-alg, \exp)$. Every semialgebraic function, as well as the logistic loss $(y,z)\mapsto \log(1+e^{-yz})$ and the hyperbolic tangent $t \mapsto \tanh(t)$, is definable in $\mathcal{R}_{\rm exp}$. 
As shown in \cite[Section~5.2]{davis2020stochastic}, \cite[Section~6.2]{bolte2021conservative}, and, more recently, \cite[Remark~3.19]{bareilles2025deep}, the loss functions arising in nearly all deep neural networks are jointly definable in suitable o-minimal structures, notably the real exponential field $\mathcal{R}_{\rm exp}$; see, e.g., \cite[Table~1]{bareilles2025deep}. Hence, under mild integrability and Lipschitz assumptions, \Cref{thm:lln} applies to the training objectives of nearly all deep neural networks.

We emphasize, however, that joint definability of $f$ does not in general imply definability of the expectation function $\Ef{f}$, which may require additional regularity assumptions on the distribution $P$; see \cite{stability11duke} and \cite{bolte2023subgradient}.

\begin{remark}[change of variable]
Joint definability requires $\Xi\subset\reals^s$ for some $s\in\nats$. However, a simple change of variable allows for greater generality. Indeed, let $Z:\Xi\to\reals^s$, and suppose that $g:\reals^s\times\reals^d\to\reals$ is jointly definable in some NIP structure. Then it can be verified that the function $
(\xi,x)\mapsto g(Z(\xi),x)$
satisfies \Cref{as:definable}. This is closely related to the uniform definability requirement in \eqref{eq:def-q}. Of course, additional integrability and Lipschitz assumptions are needed for \Cref{as:f-S}. Such a change of variable can be useful even when $\Xi=\reals$. For example, consider $
f(\xi,x)=\bm{1}_{\mathbb{Z}}(\xi)|x|,$
which is in fact separable in the sense of \Cref{as:separable} by the criterion in \Cref{prop:sep-lin}. By \cite[1.2.6]{van1998tame}, there is no o-minimal expansion of $\R-alg$ in which $f$ is jointly definable, since otherwise the set $\mathbb{Z}$ would be definable; see also \cite[Non-example 3.4]{bareilles2025deep}. However, by setting $Z(\xi)=\bm{1}_{\mathbb{Z}}(\xi)$, so that $Z(\bm{\xi})$ is a Bernoulli random variable, we may take $g(t,x)=t|x|$, which is semialgebraic. Then $f(\xi,x)=g(Z(\xi),x)$, and hence \Cref{as:definable} holds.
\end{remark}

\begin{remark}[NIP vs.~stratification]\label{rmk:omin}
Parts of the nonsmooth optimization literature invoke o-minimality to rule out pathological phenomena \cite{daniilidis2020pathological}; see, e.g., \cite{ioffe2009invitation,bolte2007clarke,davis2020stochastic,bolte2023subgradient}. These works rely primarily on geometric properties, captured by various notions of stratifiability, to establish convergence of subgradient-type methods. In \Cref{coro:j-omin}, however, the convergence of $d_X(\Enu{f},\Ef{f})$ follows from a combinatorial property, namely NIP, which is enjoyed by every o-minimal structure. By contrast, stratifiability is neither necessary nor sufficient for convergence in the Lipschitz pseudometric, even when the number of strata is uniformly bounded; see \Cref{sec:rmk:omin} for details. This distinction highlights the different roles played by combinatorial and geometric regularity.
\end{remark}

\subsubsection{Functions with piecewise uniformly definable slices}\label{sec:pu-def}

In \Cref{sec:j-def}, we examine \Cref{as:definable} only in the case $|I|=1$. Here, we show how the pieces $\{\Xi_\ell\}_{\ell\in I}$ can be used to accommodate more general functions.

 To begin, we note that these pieces naturally provide the flexibility needed to model problems with a scenario structure, allowing different scenarios to be definable in different structures.

\begin{example}[scenario formulation]\label{ex:scenario}
Let $(\Xi,\mathcal{A})$ be a measurable space, and consider a disjoint decomposition
$\Xi=\sqcup_{\ell\in I}\Xi_\ell$, where $I\subset\nats$ and $\Xi_\ell\in\mathcal{A}$ for each $\ell\in I$. Each set $\Xi_\ell$ represents the $\ell$th scenario for the random variable $\bm{\xi}$. Suppose that, for each $\ell\in I$, the function $
(\xi,x)\mapsto f(\xi,x) + \iota_{\Xi_\ell}(\xi)$
is jointly definable in an NIP structure $\mathcal{M}_\ell$. Then \Cref{as:definable} holds for $f$.
\end{example}

For \Cref{as:definable}, a natural question is how to choose the pieces $\{\Xi_\ell\}_{\ell\in I}$. While this choice may be evident in a scenario formulation, it is less apparent in general. In the sequel, we show that, for a broad class of functions, there is a universal construction that exploits the countability of the languages of the underlying structures. Our result relies on the following assumption, which is based on the notion of a \emph{totally Borel} structure; see \cite[Definition 1.1.1]{steinhorn1985chapter}.

\begin{assumption}[slicewise definability]\label{as:p-def}
	Given a probability space $(\Xi, \mathcal{A}, P)$ and a function $f:\Xi \times \reals^d \to \reals$, assume that there exists a family $\{\mathcal{M}_j\}_{j \in J}$ of NIP expansions  of $\R-alg$ indexed by $J \subset \nats$, such that, for any $j \in J$,
	\begin{enumerate}[label=\textnormal{(\roman*)}]
		\item the language of $\mathcal{M}_j$ is countable, and\label{item:as:p-def-1}
		\item $\mathcal{M}_j$ is totally Borel, meaning that every set definable in $\mathcal{M}_j$ is Borel.\label{item:as:p-def-2}
	\end{enumerate}
Moreover, for $P$-a.e.~$\xi \in \Xi$, $f(\xi, \cdot)$ is definable in $\mathcal{M}_j$ for some $j \in J$.
\end{assumption} 

In \Cref{as:p-def}, neither an explicit cover $\{\Xi_\ell\}_{\ell\in I}$ nor the uniform definability requirement in \eqref{eq:def-q} of \Cref{as:definable} is imposed. Instead, for $P$-a.e.~$\xi\in\Xi$, the slice $f(\xi,\cdot)$ is assumed to be definable in a totally Borel NIP structure $\mathcal{M}_j$ with a countable language, where $j\in J$ may depend on $\xi$; no explicit uniformity in $\xi$ is required.
We note that the change of notation from $I$ to $J$ for indexing the structures $\{\mathcal{M}_j\}_{j\in J}$ is deliberate. As will become clear below, even the case $|J|=1$ in \Cref{as:p-def} may lead to countably infinite $I$ in \Cref{as:definable}.

\begin{theorem}[slicewise definability]\label{thm:o-minimal}
Let $(\Xi, \mathcal{A})$ be a measurable space, and let $\bm{\xi},\bm{\xi}^1,\bm{\xi}^2,\ldots$ be $\Xi$-valued iid random variables on a complete probability space $(\Omega,\mathcal{F},\mathbb{P})$ with a law $P$. Suppose that $f:\Xi\times\reals^d\to\reals$ and $X\subset\reals^d$ satisfy Assumptions~\ref{as:f-S} and \ref{as:p-def}. Then
\[
    \lim_{\nu\to\infty} d_X(\Enu{f},\Ef{f})=0,
    \qquad \mathbb{P}\text{-a.s.}
\]
\end{theorem}

The proof of \Cref{thm:o-minimal} mainly verifies \Cref{as:definable} and is deferred to \Cref{sec:prf:thm:o-minimal}. 
Informally, we enumerate all formulas arising from the languages of the structures $\{\mathcal{M}_j\}_{j\in J}$, keeping track of the corresponding structure, and denote them by $\{\phi_\ell\}_{\ell\in I}$.
 Define $\Xi_\ell \subset \Xi$ with $\ell \in I$ as the set of $\xi$ whose value sequence $\{f(\xi,q_k)\}_{k \in \nats}$ can be represented by the $\ell$th formula, where $Q=\{q_k\}_{k\in\nats}$ is a countable dense subset of $X$. This enumeration is possible because $J$ is countable and each $\mathcal{M}_j$ has a countable language.  The index set $I$ may be countably infinite even if $|J|=1$. The proof then concludes by showing that each $\Xi_\ell$ is universally measurable relative to $\mathcal{A}$, using the totally Borel hypothesis and countability of $Q$.

\begin{remark}[countable language]
In \Cref{thm:o-minimal}, two countability assumptions are imposed: The index set $J$ is countable, and each structure $\mathcal{M}_j$ has a countable language. Our proof relies essentially on both assumptions to enumerate all formulas arising from the family $\{\mathcal{M}_j\}_{j\in J}$. It is natural to ask whether either assumption can be removed.  
We conjecture that neither assumption can be relaxed without additional hypotheses.
\end{remark}

As shown in \cite[Proposition~1.1]{kaiser2012first}, every o-minimal structure is totally Borel. Since every o-minimal structure is also NIP, any o-minimal structure in a countable language is admissible for \Cref{as:p-def}. Important examples include $\R-alg$ and $\mathcal{R}_{\rm exp}$ introduced in \Cref{sec:pre-model}.
 The restricted analytic field $\mathcal{R}_{\rm an}$ defines globally subanalytic functions. Although its language contains a symbol for every restricted analytic function and is therefore uncountable, most functions arising in practice involve only countably many such functions across all slices. Retaining only these functions together with the ordered-field language yields a countable-language reduct of $\mathcal{R}_{\rm an}$, which is o-minimal.

\begin{corollary}[countable-language o-minimal slices]\label{coro:o-minimal}
Let $(\Xi, \mathcal{A})$ be a measurable space, and let $\bm{\xi},\bm{\xi}^1,\bm{\xi}^2,\ldots$ be $\Xi$-valued iid random variables on a complete probability space $(\Omega,\mathcal{F},\mathbb{P})$ with a law $P$. Let $\mathcal{M}$ be a countable-language o-minimal expansion of $\R-alg$. Suppose that $f:\Xi\times\reals^d\to\reals$ and $X\subset\reals^d$ satisfy \Cref{as:f-S} and, for $P$-a.e.~$\xi \in \Xi$, the function $f(\xi, \cdot)$ is definable in the structure $\mathcal{M}$. Then
\[
    \lim_{\nu\to\infty} d_X(\Enu{f},\Ef{f})=0,
    \qquad \mathbb{P}\text{-a.s.}
\] 
\end{corollary}
\begin{proof}
	The claim follows from \cite[Proposition 1.1]{kaiser2012first}, \Cref{prop:om-NIP}, and \Cref{thm:o-minimal}.
\end{proof}

When $I$ in \Cref{as:definable} is countably infinite, \Cref{prop:norate} shows that the convergence $d_X(\Enu{f},\Ef{f})\to 0$ guaranteed by \Cref{thm:lln} can be arbitrarily slow. Although \Cref{coro:o-minimal} considers only a single structure $\mathcal{M}$, it may still give rise to a countably infinite index set $I$ in \Cref{as:definable}. Indeed, in the construction of \Cref{prop:norate}, the function $f(\xi,\cdot)$ is piecewise affine for every $\xi\in\Xi$. Hence, the convergence in \Cref{coro:o-minimal} can also be arbitrarily slow.

Before turning to the proofs, we record an example illustrating the applicability of \Cref{coro:o-minimal} to a nontrivial geometric problem, which may have facility location applications.

\begin{example}[$k$-median for sets]
Let $\bm{S}:\Omega\rightrightarrows\reals^d$ be a random set taking nonempty polytope values on a complete probability space $(\Omega,\mathcal{F},\mathbb{P})$. We seek $k$ points $x_1, \ldots, x_k \in \reals^d$ that minimize, in expectation, the distance from the closest point $x_i$ to $\bm{S}(\omega)$. For any nonempty compact set $A\subset\reals^d$, define
\[
f(A,x_1, \ldots, x_k)
=
\min_{1\leq i\leq k}\dist(x_i, A).
\]
The goal is then to find an optimal configuration $x_1, \ldots, x_k$ minimizing $
\Ex[f(\bm{S},x_1, \ldots, x_k)].$
However, since the first argument of $f$ is itself a set, $f$ is not directly jointly definable in any one-sorted structure in the sense considered here. Moreover, it is unclear whether any change of variables could remedy this issue.
For every nonempty polytope $A\subset\reals^d$, the function $f(A,\cdot)$ is semialgebraic and $1$-Lipschitz. Hence, subject to an integrability requirement, \Cref{coro:o-minimal} applies.
\end{example}

\subsection{Proofs}
We collect the deferred proofs and details from \Cref{sec:sep,sec:def}.

\subsubsection{Proof of \Cref{thm:o-minimal}}\label{sec:prf:thm:o-minimal}
Since the structure $\mathcal{M}_j$ has a countable language by \Cref{as:p-def}\ref{item:as:p-def-1}, the set of formulas in the language of $\mathcal{M}_j$ is countable; see \cite[Proposition 1.3.4]{chang1990model}. Hence, the set of all possible formulas in the structures $\{\mathcal{M}_j\}_{j \in J}$ is also countable. We can therefore enumerate all relevant formulas that have at least $d+1$ variables as $\{\phi_\ell(x,r;z_\ell)\}_{\ell \in \nats}$, 
where $x \in \reals^d, r\in\reals, z_\ell \in \reals^{m_\ell}$, $m_\ell \in \{0\}\cup \nats$, and $\phi_\ell$ is defined in the language of $\mathcal{M}_{j(\ell)}$ with $j(\ell) \in J$.
Let $Q=\{q_k\}_{k\in\nats}$ be a countable dense subset of $X$.
We then verify \Cref{as:definable} and the conclusion follows from \Cref{thm:lln}.

To this end, for each $\ell\in\nats$,  we focus on the formula $\phi_\ell$ and let
\[
Z_\ell=\bigcap\nolimits_{k=1}^\infty \{z\in \reals^{m_\ell} \mid \mathcal{M}_{j(\ell)}\models (\exists ! r)\ \phi_\ell(q_k, r; z) \},
\]
where the quantifier $\exists!$ denotes unique existence and is definable in first-order logic.\footnote{In fact, one has $(\exists! x) V(x) \iff (\exists x)(V(x)\wedge ((\neg \exists y)(V(y)\wedge (y\neq x) )))$.}
Hence, the set $\{z\in \reals^{m_\ell} \mid \mathcal{M}_{j(\ell)}\models(\exists ! r)\ \phi_\ell(q_k, r; z) \}$ is definable in  $\mathcal{M}_{j(\ell)}$. Since $\mathcal{M}_{j(\ell)}$ is totally Borel by \Cref{as:p-def}\ref{item:as:p-def-2}, $Z_\ell$ is Borel.  For any $\ell,k \in \nats$, define $G_{\ell,k}=\{(v, z) \in \reals^\nats \times Z_\ell \mid \mathcal{M}_{j(\ell)} \models \phi_\ell (q_k, v_k, z)\}$ and
\[
H_\ell = \bigcap\nolimits_{k=1}^\infty  G_{\ell,k} = \Big\{(v, z) \in \reals^\nats \times Z_\ell \Bigm| \mathcal{M}_{j(\ell)} \models \phi_\ell (q_k, v_k, z), \forall k \in \nats\Big\},
\]
where the space $\reals^\nats$ is endowed with the product topology.
Let $E_{\ell,k}=\{(v_k, z) \in \reals \times \reals^{m_\ell} \mid \mathcal{M}_{j(\ell)} \models \phi_\ell(q_k,v_k,z)\}$, which is definable in $\mathcal{M}_{j(\ell)}$, and therefore is Borel.
 The set $G_{\ell, k}$ is the preimage of a Borel set $E_{\ell,k}\cap (\reals \times Z_\ell)$ under the continuous coordinate projection $(v,z)\mapsto (v_k,z)$, so that $G_{\ell, k}$, and also $H_\ell$, are Borel on $\reals^\nats \times \reals^{m_\ell}$. Since $\reals^\nats$ and $\reals^{m_\ell}$ are Polish, \cite[Proposition 8.4.4]{cohn2013measure} implies that the following set is universally measurable relative to $\mathcal{B}(\reals^\nats)$:
\[
V_\ell = \{v \in \reals^\nats \mid \exists z \in Z_\ell, (v, z) \in H_\ell\},
\]
which, simply put, is the collection of all value sequences on $Q$ realized by the formula $\phi_\ell$ with parameters that define a unique value at every $q_k$. 
Consider a mapping $\Psi:\Xi  \to \reals^\nats$ with
	\[ 
		\Psi(\xi)= (f(\xi, q_1), f(\xi, q_2), \ldots),
	\]
	which is $\mathcal{A}/\mathcal{B}(\reals^\nats)$-measurable since $f(\cdot, x)$ is $\mathcal{A}$-measurable for all $x \in \reals^d$; see \cite[Proposition 2.4]{folland1999real}.
	For each $\ell \in \nats$, define
	\[
	\Xi_\ell=\Psi^{-1}(V_\ell) \subset \Xi.
	\]
By \cite[Lemma 8.4.6]{cohn2013measure}, the set $\Xi_\ell$
is universally measurable relative to $\mathcal{A}$; hence $\Xi_\ell \in \mathcal{A}^P$.

By construction, for any $\ell \in \nats$ and any $\xi \in \Xi_\ell$, there exists $z \in Z_\ell$ such that $\mathcal{M}_{j(\ell)} \models \phi_\ell(q_k, f(\xi, q_k); z)$ for all $k \in \nats$. Meanwhile, by construction of $Z_\ell$, we have $\mathcal{M}_{j(\ell)} \models \phi_\ell(q_k, r; z)$ if and only if $r = f(\xi, q_k)$.

 From \Cref{as:p-def}, there exists a set $\Xi_0$ with $P(\Xi_0)=1$ such that, for any $\xi \in \Xi_0$, the function $f(\xi, \cdot)$ is definable in some structure in $\{\mathcal{M}_j\}_{j \in J}$. By the definition of the sets $\{\Xi_\ell\}_{\ell\in\nats}$, every $\xi\in\Xi_0$ belongs to $\Xi_\ell$ for some $\ell\in\nats$. Hence, $\Xi_0 \subset \cup_\ell \Xi_\ell$ and $P(\cup_\ell \Xi_\ell) = 1$.

\subsubsection{Missing details in \Cref{rmk:omin}}\label{sec:rmk:omin}
Let $\mathcal{T}_S(x)$ denote the tangent cone to a set $S$ at $x$. When $S$ is a $C^1$ manifold, $\mathcal{T}_S(x)$ coincides with the tangent space to $S$ at $x$; see \cite[Example~6.8]{VaAn}.
We begin by recalling the definition of Whitney $C^p$-stratification used in \cite{bolte2007clarke,davis2020stochastic}.
\begin{definition}
A Whitney $C^p$-stratification of $S\subset\reals^d$ is a partition of
$S$ into finitely many nonempty $C^p$-submanifolds $(S_1, \ldots, S_n)$, called
strata, satisfying the following conditions.
\begin{enumerate}[label=\textnormal{(\roman*)}]
\item \textnormal{{(Frontier)}} For $i \neq j$, if
$S_i\cap \cl S_j\neq \emptyset$, then
$S_i\subset \cl S_j$.
\item \textnormal{{(Whitney-$(a)$)}} Fix $i \neq j$. If $S_i \ni x_k\to y\in S_j$ and $\mathcal{T}_{S_i}(x_k)\to T$, then
$\mathcal{T}_{S_j}(y)\subset T$.
\end{enumerate}
\end{definition}

Here, the convergence of tangent spaces is understood in the Grassmannian topology.
We next show that stratifiability is neither necessary nor sufficient for convergence in the Lipschitz pseudometric. Since one may take $f(\xi,x)=g(x)$, where $g$ is deterministic and essentially arbitrary, it is immediate that stratifiability is not necessary for the convergence of $d_X(\Enu{f},\Ef{f})$. As shown below, stratifiability is not sufficient either, even when the number of strata is uniformly bounded. We focus on the construction in \cite[Theorem~3]{tian2025failure} with $X = [0,1]^2$ and $f:\Xi \times \reals^2 \to \reals$, which, by an argument similar to that in \Cref{prop:counterex}, yields a failure of the Lipschitzian SLLN. It remains only to verify that the construction in \cite[Section~3.4]{tian2025failure} is Whitney $C^2$-stratifiable.

We first recall the construction of $f:\Xi \times \reals^2 \to \reals$ in \cite[Theorem 3]{tian2025failure} below:
\[
f(\xi, x) = \max\{g(\xi, x), 0\} + 35\|x\|^2_2,
\]
where $g(\xi, x) = x_1 + q(\xi, x_2)$ with $q:\Xi \times \reals \to \reals$. Moreover, for any $\xi \in \Xi$, the function $q(\xi, \cdot)$ is $C^{1,1}$ on $\reals$ and $C^2$ on $\reals\setminus\{0\}$; see \cite[Section 3.4]{tian2025failure}. 
Without loss of generality, we remove the quadratic term and focus on $r(\xi, x) = \max\{g(\xi, x), 0\}$.
Let $I_- = (-\infty,0), I_0=\{0\}$, $I_+ = (0,\infty)$, and a piecewise affine $h:\reals^2 \to \reals$ be defined by $h(y_1,y_2) = \max\{y_1, 0\}$. 
Then, it is easy to verify that $\gph h$ admits a Whitney $C^\infty$-stratification. In particular, $\gph h$ is a disjoint union of strata $\mathcal{P}=\{P_{\alpha, \beta}\}_{\alpha,\beta \in \{-,0,+\}}$, where
\[
P_{\alpha, \beta}
=
\{(y_1, y_2, h(y_1, y_2)) \mid y_1 \in I_\alpha,\ y_2 \in I_\beta\}.
\]
For any $\xi \in \Xi$, define the following mapping $\Phi_\xi:\reals^3 \to \reals^3$ as
\[
\Phi_\xi(x_1,x_2,t) = (x_1 + q(\xi, x_2), x_2, t),
\]
which is a $C^1$-diffeomorphism. Observe that $\Phi_\xi(\gph r(\xi, \cdot)) = \gph h$ and $\Phi_\xi$ is $C^2$ on $\reals \times (\reals \setminus\{0\})\times \reals$. For any $\xi \in \Xi$, consider the collection $\mathcal{Q}_\xi= \{Q_{\xi, \alpha, \beta}\}_{\alpha,\beta \in \{-,0,+\}}$, where
\[
Q_{\xi, \alpha, \beta}
=
\{(x_1, x_2, r(\xi, x_1, x_2)) \mid
x_1+q(\xi,x_2) \in I_\alpha,\ x_2 \in I_\beta\}
=
\Phi_\xi^{-1}(P_{\alpha,\beta}).
\]
The collection $\mathcal{Q}_\xi$ forms a finite
partition of $\gph r(\xi, \cdot)$ into connected, locally closed submanifolds,
and the frontier condition is inherited from the stratification of $h$ because
$\Phi_\xi$ is a homeomorphism. The collection $\mathcal{Q}_\xi$ consists of $C^2$-submanifolds and has at most nine elements for every $\xi \in \Xi$. Thus, $\mathcal{Q}_\xi$ is a finite $C^2$-stratification.

It remains to check Whitney-$(a)$. Fix two distinct indices
$(\alpha,\beta)$ and $(\alpha',\beta')$. The only nontrivial case is when $
 Q_{\xi,\alpha,\beta}\subset \cl Q_{\xi,\alpha',\beta'}$,
or equivalently, since $\Phi_\xi$ is a homeomorphism, $
 P_{\alpha,\beta}\subset \cl P_{\alpha',\beta'}.$
Let $V_{\alpha,\beta}$ and $V_{\alpha',\beta'}$ denote the constant tangent
spaces of $P_{\alpha,\beta}$ and $P_{\alpha',\beta'}$, respectively. Since
$\mathcal P$ is the Whitney stratification of $\gph h$, we have $
 V_{\alpha,\beta}\subset V_{\alpha',\beta'}.$
Take any sequence
$z_k\in Q_{\xi,\alpha',\beta'}$ such that
$z_k\to z\in Q_{\xi,\alpha,\beta}$, and suppose $
 \mathcal{T}_{Q_{\xi,\alpha',\beta'}}(z_k)\to T.$
Write $z_k=\Phi_\xi^{-1}(a_k)$ and $z=\Phi_\xi^{-1}(b)$, where
$a_k\in P_{\alpha',\beta'}$ and $b\in P_{\alpha,\beta}$. Since
$\Phi_\xi$ is a homeomorphism, $a_k\to b$. Moreover, since
$\Phi_\xi^{-1}$ is $C^1$,
\[
 \mathcal{T}_{Q_{\xi,\alpha',\beta'}}(z_k)
 =
 \nabla\Phi_\xi^{-1}(a_k)V_{\alpha',\beta'}
 \to
 \nabla\Phi_\xi^{-1}(b)V_{\alpha',\beta'},
\]
where the equality follows from \cite[Exercise 6.7]{VaAn}.
Hence $
 T=\nabla\Phi_\xi^{-1}(b)V_{\alpha',\beta'}.$
On the other hand, $
 \mathcal{T}_{Q_{\xi,\alpha,\beta}}(z)
 =
 \nabla\Phi_\xi^{-1}(b)V_{\alpha,\beta}.$
Because $\nabla\Phi_\xi^{-1}(b)$ is an invertible linear map and
$V_{\alpha,\beta}\subset V_{\alpha',\beta'}$, we obtain
\[
 \mathcal{T}_{Q_{\xi,\alpha,\beta}}(z)
 =
 \nabla\Phi_\xi^{-1}(b)V_{\alpha,\beta}
 \subset
 \nabla\Phi_\xi^{-1}(b)V_{\alpha',\beta'}
 =
 T.
\]
This verifies Whitney-$(a)$.
\section{Applications}\label{sec:appl}

In this section, we discuss the consequences and properties of the Lipschitz pseudometric convergence, especially its implications for other notions of convergence.

\subsection{Uniform convergence of function values}

Uniform convergence of function values ensures consistency of global minimizers; see, e.g., \cite[Section~9.2.6]{shapiro2021lectures}. As a warm-up, we observe that the Lipschitz pseudometric controls the supremum norm of the difference between two functions on a bounded set

\begin{proposition}\label{prop:ucvgt}
	Let $h,g :\reals^d \to \Reals$ be real-valued on bounded $X\subset\reals^d$ with $0 \in X$. One has
\[
\nsup_{x \in X}| h (x) - g(x) | \leq (1+\nsup_{z \in X}\|z\|)d_X(h,g).
\]
\end{proposition}
\begin{proof}
	For any $x \in X$, we have
\[
|h(x)-g(x)| \leq |h (0)-g(0)| + \glip_X(h-g)\|x\| \leq (1+\|x\|)d_X(h,g).
\]
The claim follows by taking the supremum over $X$.
\end{proof}

The boundedness requirement on $X$ in \Cref{prop:ucvgt} cannot be removed. To see it, let $d=1$, $X=\reals$, $g^\nu(x)=\frac{1}{\nu}|x|$, and $g\equiv 0$. Then, one has $
d_X(g^\nu,g)=\frac{1}{\nu}\to 0,$
whereas $
\sup_{x\in X}|g^\nu(x)-g(x)|=\infty$
for every $\nu\in\nats$.
Applying \Cref{prop:ucvgt} with $h=\Enu{f}$ and $g=\Ef{f}$ shows that a Lipschitzian SLLN implies a uniform SLLN for function values on a bounded set.

Uniform convergence of function values holds under much milder assumptions than convergence in the Lipschitz pseudometric; e.g., \Cref{as:f-S} with compact $X$ suffices \cite[Theorem~9.60]{shapiro2021lectures}. Below, we show that convergence in the Lipschitz pseudometric yields substantially stronger convergence properties than convergence of function values alone.

\subsection{Uniform convergence of subdifferentials}\label{sec:subd}

We now examine the uniform convergence of subdifferentials.
For $h:\reals^d\to\Reals$ that is Lipschitz on an open $X\subset \reals^d$, several notions of subdifferential are available; see \cite[Chapter~8]{VaAn}. In particular, the regular, limiting, and Clarke subdifferentials of $h$ at $x \in X$ can be written as
\[
\begin{aligned}
\widehat{\partial} h(x)
&=
\{
v\in\reals^d \mid 
h(x_k)\ge h(x)+\langle v,x_k-x\rangle
+o(\|x_k-x\|)
\ \text{as } x_k\to x
\},
\\
\partial h(x)
&=
\{
v\in\reals^d \mid
\exists\, x_k\to x,\ h(x_k)\to h(x),\
\exists\, v_k\to v,\ v_k\in\widehat{\partial} h(x_k)
\},\\
\overline{\partial} h(x)&=\con \partial h(x).
\end{aligned}
\]

The regular subdifferential lacks robust calculus rules and outer semicontinuity \cite[p.~301]{VaAn}. In particular, even if $x^\nu\to x$ and $0\in\widehat{\partial}h(x^\nu)$ for all $\nu \in \nats$, it need not follow that $0\in\widehat{\partial}h(x)$.
Hence, we turn to the other two notions and show that the Hausdorff distances between the corresponding limiting and Clarke subdifferentials are uniformly controlled by the global Lipschitz modulus of the difference between the two functions.

\begin{proposition}[subdifferentials]\label{prop:subd}
Let $h,g:\reals^d\to\Reals$ be Lipschitz on open $X\subset\reals^d$ with $0 \in X$. Let $D \subset X$ be the set of points where $h$ and $g$ are both differentiable. Then, 
\[
\sup_{x \in X}\setd\big(\overline{\partial} h (x), \overline{\partial} g(x)\big) \leq \sup_{x \in  X} \setd\big(\partial h (x), \partial g(x)\big) \leq \sup_{x \in D} \|\nabla (h-g)(x)\|=  \sup_{x \in X}\lip{} (h-g)(x) \leq \glip_X(h-g).
\]
If $h$ and $g$ are subdifferentially regular on $X$, then the first two inequalities hold with equality. If, in addition, $X$ is convex, then all the inequalities hold with equality.
\end{proposition}
\begin{proof}
The first inequality follows from \cite[Proposition 1.4]{iusem2010distances}. Indeed,  for any $x \in X$, we have $\setd\big(\overline{\partial} h (x), \overline{\partial} g(x) \big) \leq \setd\big(\partial h (x), \partial g(x) \big)$. 
Before proving the second inequality, we establish the equality. By Rademacher's theorem \cite[Theorem~9.60]{VaAn}, the set $D$ has full Lebesgue measure in $X$.
Let $f=h-g$. Then $f$ is Lipschitz on $X$ and $\nabla f(x) = \nabla h(x) - \nabla g(x)$ for any $x \in D$. Using \cite[Theorems~9.13~and~9.61]{VaAn}, for any $y\in X$, one has 
\[
\lip f(y)=\nsup_{v \in \partial f(y)}  \|v\| = \nsup_{v \in \overline{\partial}f(y)}  \|v\| \leq \nsup_{x \in D} \|\nabla f(x)\|= \nsup_{x \in D} \|\nabla h(x) - \nabla g(x)\|.
\]
Hence, we obtain one direction of the equality. For the reverse direction, note that, by \cite[Theorem~9.13]{VaAn}, $\|\nabla f(x)\| \leq \lip f(x)$ for every $x\in D$, which yields the desired equality. We now go back to the second inequality. Fix any $x \in X$. By \cite[Corollary~10.9]{VaAn} and the fact that $h,g$ are Lipschitz on $X$, one has $\partial h(x) \subset \partial g(x) + \partial f(x)$. 
Moreover, by \cite[Theorem~9.13]{VaAn}, $\partial f(x)\subset \lip f(x)\ball$, and hence $\exs(\partial h(x);\partial g(x))\leq \lip f(x)$. Interchanging $h$ and $g$ and taking the supremum over $x\in X$ yields the claim. 
The final inequality follows directly from the definition and becomes an equality when $X$ is convex \cite[Theorem~9.2]{VaAn}. If $h$ and $g$ are subdifferentially regular, then the first inequality is an equality by \cite[Corollary~8.11]{VaAn}, while the second one follows from $D\subset X$ and $\{\nabla h(x)\}=\partial h(x)$, $\{\nabla g(x)\}=\partial g(x)$ for every $x\in D$.
\end{proof}

The openness requirement on $X$ in \Cref{prop:subd} cannot be removed; see the discussion in \cite[Remark~A]{ruan2024subgradient}. Since the set $X$ in \Cref{prop:subd} need not be bounded, applying \Cref{prop:subd} to $\Enu{f}$ and $\Ef{f}$, together with the results of \Cref{sec:lslln}, yields uniform SLLNs for subdifferentials over possibly unbounded sets. When $f(\xi,\cdot)$ is subdifferentially regular on $X$ for every $\xi$ and $X$ is convex, \Cref{prop:subd} and the SLLN imply, under \Cref{as:f-S}, that the uniform SLLN for subdifferentials is equivalent to the Lipschitzian SLLN. Thus, the Lipschitz pseudometric provides a tight surrogate for studying uniform convergence of subdifferentials.

By \Cref{prop:subd}, the Lipschitzian SLLN discussed above implies uniform SLLNs for subdifferentials. For illustration, we record below one concrete instance and refer the reader to \Cref{sec:lslln} for more general results.

\begin{corollary}[uniform convergence of subdifferentials]\label{coro:ulln-subd}
Let $(\Xi,\mathcal{A})$ be a measurable space and $\bm{\xi}, \bm{\xi}^1,\bm{\xi}^2,\ldots$ be $\Xi$-valued iid random variables on the complete probability space $(\Omega,\mathcal{F},\mathbb{P})$ with a law $P$.
Suppose that, for $f:\Xi\times\reals^d\to\reals$ and an open $X\subset\reals^d$ with $0 \in X$, the following hold.
\begin{enumerate}[label=\textnormal{(\roman*)}]
	\item For any $x \in \reals^d$, the function $f(\cdot, x)$ is $\mathcal{A}$-measurable and $\Ex[|f(\bm{\xi}, 0)|] < \infty$.
	\item There exists an $\mathcal{A}$-measurable $L:\Xi \to [0,\infty)$ with $\Ex L < \infty$ such that, for $P$-a.e.~$\xi \in \Xi$, the function $f(\xi, \cdot)$ is $L(\xi)$-Lipschitz  on $X$.
\end{enumerate}
Moreover, suppose one of the following three conditions holds.

\begin{enumerate}[label=\textnormal{(\roman*)}, start=3]
	\item The collection $\{f(\xi, \cdot)|_X\mid \xi \in \Xi\}$ is essentially separable in $(\Lip(X), \|\cdot\|_{\lip})$.
	\item The function $f$ is jointly definable in an o-minimal expansion of $\R-alg$.\label{item:coro:ulln-subd-2}
	\item For $P$-a.e.~$\xi \in \Xi$, the function $f(\xi, \cdot)$ is definable in the structure $\mathcal{R}_{\exp}$.
\end{enumerate}	
	Then, $\mathbb{P}$-a.s., one has
\[
\lim_{\nu \to \infty}\nsup_{x \in X}\setd\big(\partial \Enu{f} (x), \partial \Ef{f}(x)\big) = 0 \quad\text{and}\quad
\lim_{\nu \to \infty}\nsup_{x \in X}\setd\big(\overline\partial \Enu{f} (x), \overline\partial \Ef{f}(x)\big) = 0.
\]
If \ref{item:coro:ulln-subd-2} holds and, in addition, $
\Ex[|f(\bm{\xi},0)|^2]<\infty$ and $\Ex[L^2]<\infty$, then, for every $\delta>0$, there exists $C_\delta<\infty$ such that, for all sufficiently large $\nu\in\nats$, 
\[
 \nsup_{x \in X}\setd\big(\overline\partial \Enu{f} (x), \overline\partial \Ef{f}(x)\big) \leq \nsup_{x \in X}\setd\big(\partial \Enu{f} (x), \partial \Ef{f}(x)\big) \leq  C_\delta\nu^{-1/2}.
 \]
 on some event $\Omega_{\delta,\nu} \subset \Omega$ with $\mathbb{P}(\Omega_{\delta,\nu})\geq 1-\delta$. 
\end{corollary}
\begin{proof}
	The claim follows from \Cref{coro:j-omin}, \Cref{thm:lln,thm:o-minimal}, and \Cref{prop:subd}.
\end{proof}

The uniform law in \Cref{coro:ulln-subd} has immediate algorithmic consequences. Suppose that $\Enu{f}$ is bounded below and satisfies the assumptions of \Cref{coro:ulln-subd} with $X=\reals^d$. Then existing algorithms can compute an approximate stationary point $\bm{x}^\nu$ in the following sense:
\begin{equation}\label{eq:GAS}
0 \in \con \Big(\bigcup\nolimits_{\bm{y}:\|\bm{y} - \bm{x}^\nu\| \leq  \delta^\nu} \overline\partial \Enu{f}(\bm{y})\Big) + \epsilon^\nu \ball,
\end{equation}
where $\epsilon^\nu,\delta^\nu > 0$; see, e.g., \cite{goldstein1977optimization,burke2020gradient} and references therein. 
When $\epsilon^\nu\downarrow 0$, $\delta^\nu\downarrow 0$, and $\bm{x}^\nu\to\bm{x}$ $\mathbb{P}$-a.s., \Cref{coro:ulln-subd}, together with the outer semicontinuity of the Clarke subdifferential, yields the desired consistency conclusion $0\in\overline{\partial}\Ef{f}(\bm{x})$ $\mathbb{P}$-a.s. This conclusion is stronger than the known consistency results for weak stationarity \cite{shapiro2007uniform,norkin-wets}.  Similar results hold for limiting subdifferentials, but become nontrivial only when a tighter stationarity notion than \eqref{eq:GAS}, without convexification, is considered.

Uniform SLLNs for Clarke subdifferentials have been studied for various classes of functions and in different forms, whereas substantially less is known for limiting subdifferentials. To place the results of \Cref{coro:ulln-subd} in perspective, we make the following remarks.

\begin{itemize}
\item The study of uniform SLLNs for Clarke subdifferentials was initiated in \cite{shapiro2007uniform}, which establishes a uniform result for enlarged subdifferentials; see \cite[Theorem~2]{shapiro2007uniform}. It was shown in \cite{norkin-wets} that, under suitable conditions, the enlarged uniform law of \cite{shapiro2007uniform} can be interpreted as a graphical law. Whether the enlargement can be removed to yield a genuine uniform law under \Cref{as:f-S} was posed as a question in \cite[Remark~2]{shapiro2007uniform} and answered negatively in our previous work \cite{tian2025failure}. More recently, \cite{le2026graphical} establishes a graphical law under general weak convergence of distributions. We emphasize that \Cref{coro:ulln-subd} does not extend to such a weak-convergence setting. Indeed, let $\bm{\xi}^\nu$ have law $P^\nu=\delta_{1/\nu}$, so that $P^\nu$ converges weakly to $P=\delta_0$. 
For $f(\xi,x)=|x-\xi|$ and $X=(-1,1)$, however, the subdifferentials of $\int f(\xi,\cdot)\,\dd P^\nu(\xi)$ fail to converge uniformly in the Hausdorff distance over $X$ to those of $\int f(\xi,\cdot)\,\dd P(\xi)$. 
We also note that, without subdifferential regularity, the results in \cite{shapiro2007uniform,norkin-wets,le2026graphical} concern only the consistency of weak stationarity.
 
\item Uniform SLLNs for Clarke subdifferentials without enlargement have nevertheless been established for several special classes of functions; see \cite[p.~449]{shapiro2021lectures} and \cite{teran2008uniform,ruan2024subgradient}. In particular, \cite{teran2008uniform} establishes a uniform law under a separability-type condition imposed directly on the subdifferential mapping, which is conceptually similar to our \Cref{as:separable}. However, unlike \Cref{coro:ulln-subd}, \cite{teran2008uniform} addresses only weak Clarke stationarity and does not extend directly to Clarke or limiting stationarity. The work \cite{ruan2024subgradient} establishes a uniform law for a class of convex-composite functions under a finite-VC-dimension assumption on the inner smooth mapping, which motivated our use of a VC framework in developing \Cref{as:definable}. The scope of \cite{ruan2024subgradient}, however, is restricted to weakly convex functions and therefore does not cover subdifferentially irregular functions, which are pervasive in deep neural networks.

\item There are also several algorithm-dependent results and smoothing-based analyses; see, e.g., \cite{ruszczynski2021stochastic,bolte2023subgradient,areces2026finding,xu2009smooth,burke2020subdifferential}. The work \cite{bolte2023subgradient} establishes convergence of a stochastic subgradient method to a stationary point of the population objective under definability-type assumptions on both the function $f$ and the distribution $P$. The analysis in \cite{bolte2023subgradient} relies crucially on the stratification guaranteed by definability. As noted in \Cref{rmk:omin}, stratifiability is neither necessary nor sufficient for our algorithm-independent results. The work \cite{areces2026finding} considers a randomized proximal-point method for finding approximate stationary points of the population problem. Their approach avoids the need for a uniform SLLN for subdifferentials, but fundamentally relies on convexity. The works \cite{xu2009smooth,burke2020subdifferential} establish consistency results for stationarity when $f$ is replaced by certain smoothed surrogates. We note that these algorithm-dependent and smoothing-based results do not establish convergence of the empirical landscape of the original SAA problem; rather, they apply to sequences of iterates generated by specific algorithms or to the empirical landscape of suitably smoothed surrogate problems.

\item The work \cite{qi2022asymptotic} studies consistency of d(irectional)-stationary points for a structured difference-of-convex problem; see \cite[Section~3]{qi2022asymptotic}. Their results for d-stationary points require a stringent separation condition that may fail even for jointly piecewise affine functions; see \cite[Corollary~5.8]{qi2022asymptotic}. Another related work, \cite{schechtman2026gradient}, studies limit fields generated by gradients, Clarke Jacobians, and more general conservative mappings. They show that, under joint definability of the parametrized family in a fixed o-minimal structure, the limit field admits a variational stratification. As noted in \cite[Remark~4.6]{schechtman2026gradient}, slicewise definability alone is insufficient for their conclusion. In contrast, our \Cref{thm:o-minimal} requires only slicewise definability of $f$, provided that the relevant structures have countable languages.
\end{itemize}

\subsection{Finite-sample identification of solutions}

It is well known that graphical convergence of subdifferentials already suffices to ensure consistency of stationary points; see, e.g., \cite{norkin-wets}. This is weaker than the uniform convergence of subdifferentials in \Cref{coro:ulln-subd}. In this section, we study finite-sample identification of global minimizers and stationary points, for which graphical convergence alone is insufficient.

Under \Cref{as:f-S} with closed $X$, $\mathbb{P}$-a.s., every cluster point of a sequence of global minimizers of the SAA problems in \eqref{eq:saa} is a global minimizer of the population problem in \eqref{eq:primal}; see \cite[Theorem~2.3]{artstein1995consistency}. A substantially stronger statement, termed \emph{finite-sample identification}, is that, for all sufficiently large $\nu\in\nats$,
\begin{equation}\label{eq:id-glob}
\nargmin_{x\in X}\Enu{f}(x)
\subset
\nargmin_{x\in X}\Ef{f}(x),
\qquad
\mathbb{P}\text{-a.s.}
\end{equation}
In other words, for all sufficiently large but finite sample sizes, every global minimizer of the SAA problem, whenever one exists, is already an exact global minimizer of the population problem. This is stronger than asymptotic consistency of sequences of SAA global minimizers.
Several sufficient conditions for finite-sample identification have been proposed in \cite{shapiro2000rate,kleywegt2002sample,shapiro2002conditioning}; see also the discussion in \cite[Section~5.3.3]{shapiro2021lectures}. 
A central assumption in these works is the sharpness of the population problem, which is closely related to the notion of weak sharp minima in \cite{burke1993weak}. Formally, we say that the minimizer set of $\Ef{f}$ over $X$ is \emph{sharp} if there exists $c>0$ such that
\[
\Ef{f}(y)
\geq
\ninf_{x\in X}\Ef{f}(x)
+
c\,\dist(y,\nargmin_{x\in X}\Ef{f}(x)),
\qquad
\forall y\in X;
\]
see \cite[Equation~5.138]{shapiro2021lectures} and \cite[Assumption~(A)]{shapiro2000rate}.
Another essential ingredient is a uniform convergence result for convex subdifferentials \cite[Lemma~2.4]{shapiro2000rate}, which fails in general \cite{tian2025failure}. Consequently, existing results \cite{shapiro2000rate,kleywegt2002sample,shapiro2002conditioning} typically impose strong assumptions, 
such as requiring either that the population problem have a unique minimizer, or that $X$ be closed, convex, and polyhedral, the distribution of $\bm{\xi}$ have finite support, and $f(\xi,\cdot)$ be convex and polyhedral; see \cite[p.~184, Theorem~5.23 and (P1)--(P3)]{shapiro2021lectures}.

We now show that, if $d_X(\Enu{f},\Ef{f})\to 0$ and the sharpness condition above holds, then finite-sample identification of global minimizers follows. Combined with the Lipschitzian SLLNs in \Cref{sec:lslln}, this yields an analogous finite-sample identification result for a substantially broader class of problems, without requiring polyhedrality, finite support of the distribution, or convexity.

\begin{proposition}[identification of global minimizers]\label{prop:exact}
Suppose that $\Ef{f}$ is continuous on $X\subset \reals^d$ with $0 \in X$ and $d_X(\Enu{f}, \Ef{f}) \to 0$, $\mathbb{P}$-a.s. If the minimizer set of $\Ef{f}$ over $X$ is nonempty and sharp, then,  for all sufficiently large $\nu\in\nats$, one has 
\[
\nargmin_{x \in X} \Enu{f}(x) \subset \nargmin_{x \in X} \Ef{f}(x), \qquad \mathbb{P}\text{-a.s.}
\] 
\end{proposition}

\begin{proof}
Let $g^\nu =\Enu{f}- \Ef{f}$ and $A=\argmin_{x\in X}\Ef{f}(x)$. For any $y\in X$ and $x\in A$, one has
\[
\Enu{f}(y)-\Enu{f}(x)
=\Ef{f}(y)-\Ef{f}(x)+g^\nu(y)-g^\nu(x)
\geq c \dist(y,A)-\glip_X(g^\nu)\|y-x\|.
\]
Using $\inf_{z\in X}\Enu{f}(z)\leq \Enu{f}(x)$ and taking the supremum over $x\in A$, we obtain
\[
\Enu{f}(y)-\ninf_{z\in X}\Enu{f}(z)
\geq
(c-\glip_X(g^\nu))\dist(y,A).
\]
Since $d_X(\Enu{f}, \Ef{f}) \to 0$,  $\mathbb{P}$-a.s., then $\glip_X(g^\nu)\to0$, $\mathbb{P}$-a.s., so eventually $c>\glip_X(g^\nu)$, $\mathbb{P}$-a.s. 
Now let $y\in\argmin_{x\in X}\Enu{f}(x)$. Then the left-hand side above is zero, and therefore
$\dist(y,A)=0$. Since $\Ef{f}$ is continuous on $X$, the set
$A$ is closed relative to $X$. Thus $y\in A$.
\end{proof}

For nonconvex problems, however, finding global minimizers of the SAA problem can be difficult, motivating an analogous result for stationary points. The uniform convergence of subdifferentials in \Cref{coro:ulln-subd} yields such an identification result under a similar sharpness condition.

\begin{proposition}[stationary points]\label{coro:exact-stat}
Suppose that $\partial \Enu{f}(x)$ converges to $\partial \Ef{f}(x)$ uniformly over $x\in X$ in the Hausdorff distance, $\mathbb{P}$-a.s., and that, for a set $S \subset  X$, one has  
\begin{equation}\label{eq:stat-sep}
	\ninf_{x\in X\setminus S}
\dist\bigl(0,\partial \Ef{f}(x)\bigr)
\geq \eta
\end{equation}
for some tolerance $\eta>0$.
Then, for any $\epsilon^\nu\downarrow 0$ and all sufficiently large $\nu\in\nats$, $\mathbb{P}$-a.s., one has
\[
S^\nu=\bigl\{
x\in X
\bigm| \dist\bigl(0,\partial \Enu{f}(x)\bigr)
\leq \epsilon^\nu
\bigr\} \subset S.
\]
\end{proposition}

\begin{proof}
Let 
$\delta^\nu=
\nsup_{x\in X}
\setd\bigl(
\partial \Enu{f}(x),
\partial \Ef{f}(x)
\bigr),$ so that $\delta^\nu \to 0$, $\mathbb{P}$-a.s.
For every $y\in S^\nu$,
\[
\dist\bigl(0,\partial \Ef{f}(y)\bigr)
\leq
\dist\bigl(0,\partial \Enu{f}(y)\bigr)
+
\delta^\nu
\leq
\epsilon^\nu+\delta^\nu.
\]
For all sufficiently large $\nu \in \nats$, $\epsilon^\nu + \delta^\nu < \eta$. Hence $y\notin X\setminus S$, and therefore $y\in S$.
\end{proof}

Simply put, condition \eqref{eq:stat-sep} requires the subdifferential to be uniformly bounded away from zero outside the set $S$. This property is, in fact, part of the definition of an \emph{active manifold}; see \cite[Definition~2.6]{davis2022proximal} and \cite{lemarechal2000,lewis2002active} for the historical development. 
Hence, under the assumptions of \Cref{coro:ulln-subd}, every active manifold at a population critical point in $X$ can be locally identified by the SAA problem using finitely many samples, almost surely.

\subsection{Permanence properties}\label{sec:permanence}

We have seen in \Cref{prop:subd} that the Lipschitz pseudometric serves as a surrogate for controlling uniform convergence of subdifferentials. In this subsection, we record several operations under which convergence in the Lipschitz pseudometric is preserved. 

\begin{proposition}[permanence]\label{prop:permanence}
Suppose that, for every $i\in[m]$, the functions $h_i^\nu,h_i:\reals^d\to\reals$ are Lipschitz on $X\subset\reals^d$ with $0\in X$ and satisfy $d_X(h_i^\nu,h_i)\to0$ as $\nu\to\infty$. Then the following statements hold.
\begin{enumerate}[label=\textnormal{(\alph*)}]
\item For any $a_1,\ldots,a_m\in\reals$, one has $
d_X(
\nsum_{i=1}^m a_ih_i^\nu,
\nsum_{i=1}^m a_ih_i
)
\to 0$ as $\nu \to \infty$.\label{item:prop:permanence-1}

\item If $0\in Z\subset X$, then $
d_Z(h_i^\nu,h_i)\to 0$ as $\nu \to \infty$,
for every $i\in[m]$. \label{item:prop:permanence-2}

\item Let $0\in Y\subset\reals^p$, and let the mapping $T:Y\to X$ be Lipschitz. Then $
d_Y(h_i^\nu\circ T,h_i\circ T)\to 0$ as $\nu \to \infty$,
for every $i\in[m]$. \label{item:prop:permanence-3}

\item Let $H^\nu=(h_1^\nu,\ldots,h_m^\nu)$ and $H=(h_1,\ldots,h_m)$. When $X$ is bounded, for every $C^1$ function $g:\reals^m\to\reals$ with Lipschitz gradient, one has $
d_X(g\circ H^\nu,g\circ H)\to 0$ as $\nu \to \infty$. \label{item:prop:permanence-4}
\end{enumerate}
\end{proposition}
\begin{proof}
Let $f_i^\nu=h_i^\nu-h_i$.
Item~\ref{item:prop:permanence-1} follows from the triangle inequality and homogeneity of $\|\cdot\|_{\lip}$ in $\Lip(X)$. Item~\ref{item:prop:permanence-2} is a special case of \ref{item:prop:permanence-3}. For Item~\ref{item:prop:permanence-3}, let $L_T$ be a Lipschitz constant of $T$. Since $T(0) \in X$, one has 
\[
|f_i^\nu(T(0))|\leq |f_i^\nu(0)| + \glip_X(f_i^\nu)\|T(0)\|,\qquad\glip_Y(f_i^\nu\circ T)\leq \glip_X(f_i^\nu)L_T.
\] 
Hence, we obtain $d_Y(h_i^\nu\circ T,h_i\circ T)\leq \max\{1+\|T(0)\|,L_T\}d_X(h_i^\nu,h_i)\to0 $ as $\nu \to \infty$.
 For Item~\ref{item:prop:permanence-4}, let $\beta$ be a Lipschitz constant of $\nabla g$.  
 Let $L^\nu=\sum_{i=1}^m\glip_X(f_i^\nu)$ and $\delta^\nu=\sum_{i=1}^m\sup_{x\in X}|f_i^\nu(x)|$. One has $L^\nu \to 0$. Since $X$ is bounded, \Cref{prop:ucvgt} gives
$\delta^\nu\to0$. 
Let $C = \sup_{z \in X}\| \nabla g(H(z)) \|
 + \beta( \sum_{i=1}^m \glip_X(h_i) +2) < \infty$, since $H(X)$ is bounded.
For all sufficiently large $\nu$, one has $L^\nu\le1$.
The fundamental theorem of calculus and Lipschitz continuity
of $\nabla g$ then yield, for all $x,y\in X$,
\[
|g(H^\nu(x))-g(H(x))-g(H^\nu(y))+g(H(y))|
\leq
C(L^\nu +\delta^\nu
)\|x-y\|.
\]
Hence, $\glip_X(g\circ H^\nu -  g\circ H) \leq C(L^\nu +\delta^\nu
) \to 0$ as $\nu \to \infty$.
Moreover, $H^\nu(0)\to H(0)$ and continuity of $g$ imply
$g(H^\nu(0))-g(H(0))\to0$.
Hence $d_X(g\circ H^\nu,g\circ H)\to0$.
\end{proof}

Linear operations, restrictions, and Lipschitz changes of variables require no additional regularity, while smooth outer compositions are admissible on bounded sets. These permanence properties provide calculus rules that extend the preceding results to broader classes of problems. 

The Lipschitz pseudometric is often more convenient to work with than subdifferentials, whose calculus rules, even under suitable constraint qualifications, generally yield only outer approximations; see \cite[Chapter~10]{VaAn}. For illustration, let $h^\nu,h,g^\nu,g\in\Lip(\reals^d)$ satisfy $d_{\reals^d}(h^\nu,h)\to0$ and $d_{\reals^d}(g^\nu,g)\to0$. By \Cref{prop:subd} and the stability of the Hausdorff distance under Minkowski addition,
\[
\nsup_{x\in\reals^d}
\setd\bigl(
\partial h^\nu(x)+\partial g^\nu(x),
\partial h(x)+\partial g(x)
\bigr)
\to0.
\]
This conclusion, however, only compares the weaker first-order approximations given by the Minkowski sums of the subdifferentials, since, in general, $\partial(h+g)(x)\subset\partial h(x)+\partial g(x)$. By the calculus rules in \Cref{prop:permanence}, we instead have $d_{\reals^d}(h^\nu+g^\nu,h+g)\to0$. Therefore, \Cref{prop:subd} yields
\[
\nsup_{x\in\reals^d}
\setd\bigl(
\partial(h^\nu+g^\nu)(x),
\partial(h+g)(x)
\bigr)
\to0.
\]
Thus, owing to its favorable calculus properties, the Lipschitz pseudometric provides a convenient surrogate for controlling Hausdorff distances between subdifferentials.

So far, our uniform convergence results for subdifferentials have concerned unconstrained objectives. Let $\mathcal{N}_X(x)$ denote the limiting normal cone to a set $X$ at $x$; see \cite[Definition~6.3]{VaAn}. We next incorporate the constraint set in the population problem \eqref{eq:primal} and its SAA \eqref{eq:saa}, as illustrated by the following result.
\begin{proposition}[constrained problem]
Let $h^\nu,h : \reals^d \to \Reals$ be Lipschitz on open $Y\subset \reals^d$, and let $X\subset Y$ be closed with $0\in X$. One has
\[
\begin{aligned}
\nsup_{x\in X}
\setd\bigl(
\partial h^\nu(x)+\mathcal{N}_X(x),
\partial h(x)+\mathcal{N}_X(x)
\bigr)
& \leq d_Y(h^\nu, h), \\
\nsup_{x\in X}
\setd\bigl(
\partial (h^\nu+\iota_X)(x),
\partial (h + \iota_X)(x)
\bigr)
& \leq d_Y(h^\nu, h).
\end{aligned}
\]
The same conclusions hold for Clarke subdifferentials and Clarke normal cones.
\end{proposition}
\begin{proof}
Since the
claim is trivial when $d_Y(h^\nu,h)=\infty$, suppose that it is finite.
Then $\glip_Y (h^\nu-h)\leq d_Y(h^\nu,h)$. By
\cite[Exercise~10.10 and Theorem~9.13]{VaAn}, for every
$x\in X$, $
\partial h^\nu(x)
\subset \partial h(x)+\glip_Y(h^\nu - h)\ball
$. Interchanging $h^\nu$ and $h$ gives $\partial h(x)\subset \partial h^\nu(x)+\glip_Y(h^\nu-h)\ball$.
Adding $\mathcal{N}_X(x)$ to both inclusions and using \cite[Example~4.13, 4(5)]{VaAn} yields the first inequality.

Now let $f^\nu=h^\nu+\iota_X$ and $f=h+\iota_X$. Since $X$ is nonempty and closed and $h^\nu,h$ are continuous on $Y$, the functions $f^\nu$ and $f$ are proper and lower semicontinuous on a neighborhood of any $x\in X$. Applying the same argument as in the first claim to $f^\nu=f+(h^\nu-h)$ and $f=f^\nu+(h-h^\nu)$ yields the second inequality. The Clarke case follows by similar arguments.
\end{proof}

Therefore, a Lipschitzian SLLN on an open superset of $X$ yields uniform SLLNs for the subdifferentials of the constrained problems \eqref{eq:primal} and \eqref{eq:saa}.

As a final illustration, we show that the subdifferential of a difference-of-convex function can be uniformly controlled by those of its convex components. This observation may be of independent interest and yields an extension of the results in \cite{ruan2024subgradient} to the difference-of-convex setting.

\begin{proposition}
	Suppose that $h,h^\nu,g,g^\nu:\reals^d \to \reals$ are Lipschitz and subdifferentially regular on an open  $X \subset \reals^d$. Then, one has
	\[
	\begin{aligned}
			\sup_{x \in X}\setd\big( \partial (h-g) (x), \partial (h^\nu - g^\nu)(x)\big)&\leq \sup_{x \in X}\setd\big( \partial h (x), \partial h^\nu (x)\big) + \sup_{x \in X}\setd\big( \partial g (x), \partial g^\nu(x)\big),\\
				\sup_{x \in X}\setd\big( \partial h(x)-\partial g(x), \partial h^\nu(x) - \partial g^\nu(x)\big)&\leq \sup_{x \in X}\setd\big( \partial h (x), \partial h^\nu (x)\big) + \sup_{x \in X}\setd\big( \partial g (x), \partial g^\nu(x)\big).
	\end{aligned}
	\]
\end{proposition}
\begin{proof}
Let $f^\nu = h^\nu - g^\nu$ and $f = h-g$.
	By \Cref{prop:subd} and \cite[Exercise~9.8]{VaAn},
	\[
	\begin{aligned}
	\sup_{x \in X}\setd\big( \partial f (x), \partial f^\nu(x)\big) \leq \sup_{x \in X} \lip (f^\nu - f)(x) &\leq \sup_{x \in X} \lip (h^\nu - h)(x) + \sup_{x \in X} \lip (g^\nu - g)(x) \\
	&=\sup_{x \in X}\setd\big( \partial h (x), \partial h^\nu (x)\big) + \sup_{x \in X}\setd\big( \partial g (x), \partial g^\nu(x)\big),
	\end{aligned}
	\]
	which yields the first inequality. The second inequality follows directly from elementary properties of the Hausdorff distance; see \cite[Theorems~C.3.3.2 and~C.3.3.6]{hiriart2004fundamentals}.\end{proof}

\section{Concluding Remarks}\label{sec:conclusion}

In this paper, we establish uniform SLLNs for limiting and Clarke subdifferentials, yielding consistency of stationary points for broad classes of nonsmooth stochastic programs. To the best of our knowledge, these are the first such uniform laws that avoid interchanging expectation and subdifferentiation while allowing subdifferentially irregular random functions. Our results are new even when $f(\xi, \cdot)$ is convex for any $\xi \in \Xi$; see \cite[Theorem 3]{tian2025failure}. Moreover, under the stated integrability and Lipschitz assumptions, they apply to the training objectives of most deep neural networks without restriction to specific algorithms. In particular, we identify broad structural conditions that preclude the failures exhibited in our previous work \cite{tian2025failure}, 
suggesting that the main obstruction to obtaining a uniform law may lie not in the lack of continuity of subdifferentials \cite[Remark~2]{shapiro2007uniform}, but rather in pathological oscillatory behavior.

Conceptually, the central shift in perspective in our approach is to compare objectives through the Lipschitz pseudometric before passing to their subdifferentials. When the integrands are subdifferentially regular on an open convex set, this pseudometric provides a tight surrogate for uniform convergence of subdifferentials. For more general, possibly irregular functions, our Lipschitzian law simultaneously controls the uniform convergence of limiting and Clarke subdifferentials. As illustrated in \Cref{sec:permanence}, the Lipschitz pseudometric also enjoys favorable calculus properties, making it more convenient to work at the function level and pass to subdifferentials only afterward.

Technically, our use of definability in \Cref{as:definable} to study consistency beyond function values may be of independent interest. As emphasized in \Cref{rmk:omin}, our analysis relies on NIP, which provides a combinatorial form of control distinct from the geometric control afforded by stratifiability, a property commonly used in nonsmooth optimization \cite{ioffe2009invitation,bolte2007clarke,davis2020stochastic,bolte2023subgradient}. By exploiting the countability of the family of structures and of their languages, we show that slicewise definability already suffices for the desired Lipschitzian law, so joint definability is unnecessary. Finally, the closure properties of uniform definability suggest that the VC argument in \Cref{sec:prf-def} may extend beyond first-order difference quotients to higher-order differences and other approximation concepts, opening the door to further consistency results.

\smallskip
\paragraph*{Acknowledgment.} This work was supported in part by the Office of Naval Research under Grant N00014-24-1-2492. The authors used ChatGPT and Gemini to assist with identifying relevant references, mathematical proofreading, and editorial feedback. The authors independently retrieved, read, and verified all original sources and take full responsibility for all citations and claims in the manuscript.
	\bibliography{ref}

@book{van1998tame,
  title={Tame Topology and O-minimal Structures},
  author={van den Dries, L.},
  volume={248},
  year={1998},
  publisher={Cambridge University Press}
}

@article{laskowski1992vapnik,
  title={{V}apnik-{C}hervonenkis classes of definable sets},
  author={Laskowski, M. C.},
  journal={Journal of the London Mathematical Society},
  volume={45},
  number={2},
  pages={377--384},
  year={1992},
  publisher={London Mathematical Society}
}

@article{mei2018landscape,
  title={The landscape of empirical risk for nonconvex losses},
  author={Mei, S. and Bai, Y. and Montanari, A.},
  journal={Annals of Statistics},
  volume={46},
  number={6A},
  pages={2747--2774},
  year={2018},
  publisher={IMS}
}

@article{beer2013lipschitz,
  title={The {L}ipschitz metric for real-valued continuous functions},
  author={Beer, G. and Hoffman, M. J.},
  journal={Journal of Mathematical Analysis and Applications},
  volume={406},
  number={1},
  pages={229--236},
  year={2013},
  publisher={Elsevier}
}

@book{simon2014guide,
  title={A Guide to {NIP} Theories},
  author={Simon, P.},
  year={2015},
  series={Lecture Notes in Logic},
  publisher={Cambridge University Press}
}

@article{ruan2024subgradient,
  title={On the uniform convergence of subdifferentials in stochastic optimization and learning},
  author={Ruan, F.},
  journal={Mathematics of Operations Research},
  volume={51},
  number={3},
  pages={2323--2353},
  year={2026},
  publisher={INFORMS}
}

@article{tian2025failure,
  title={Failure of uniform laws of large numbers for subdifferentials and beyond},
  author={Tian, L. and Royset, J. O.},
  journal={arXiv preprint arXiv:2511.16568},
  year={2025}
}

@book{ledoux1991probability,
  title={Probability in Banach Spaces: Isoperimetry and Processes},
  author={Ledoux, M. and Talagrand, M.},
  volume={23},
  year={1991},
  publisher={Springer}
}

@article{shapiro2007uniform,
  title={Uniform laws of large numbers for set-valued mappings and subdifferentials of random functions},
  author={Shapiro, A. and Xu, H.},
  journal={Journal of Mathematical Analysis and Applications},
  volume={325},
  number={2},
  pages={1390--1399},
  year={2007},
  publisher={Elsevier}
}

@article{norkin-wets,
  title={On a strong graphical law of large numbers for random semicontinuous mappings},
  author={Norkin, V. I. and Wets, R. {J-B}},
  journal={Vestnik Sankt-Peterburgskogo Universiteta},
  volume={Seriya 10},
  number={3},
  pages={102--111},
  year={2013}
}

@book{clarke1990optimization,
  title={Optimization and Nonsmooth Analysis},
  author={Clarke, F. H.},
  year={1990},
  publisher={SIAM}
}

@article{teran2008uniform,
  title={On a uniform law of large numbers for random sets and subdifferentials of random functions},
  author={Ter{\'a}n, P.},
  journal={Statistics \& Probability Letters},
  volume={78},
  number={1},
  pages={42--49},
  year={2008},
  publisher={Elsevier}
}

@article{davis2022graphical,
  title={Graphical convergence of subgradients in nonconvex optimization and learning},
  author={Davis, D. and Drusvyatskiy, D.},
  journal={Mathematics of Operations Research},
  volume={47},
  number={1},
  pages={209--231},
  year={2022},
  publisher={INFORMS}
}

@book{shapiro2021lectures,
  title={Lectures on Stochastic Programming: Modeling and Theory},
  author={Shapiro, A. and Dentcheva, D. and Ruszczy\'{n}ski, A.},
  year={2021},
  edition = {3rd},
  publisher={SIAM}
}

@book{vanderVaartWellner.23,
	Author = {{van der} Vaart, A. W.  and Wellner, J. A.},
	Publisher = {Springer},
	Series = {},
	Title = {Weak Convergence and Empirical Processes},
	Edition = {2nd},
	Year = {2023}
}

@book{VaAn,
	Author = {{R. T.} Rockafellar and R. {J-B} Wets},
	Edition = {3rd printing-2009},
	Publisher = {Springer},
	Title = {Variational Analysis},
	Year = {1998}}

@article{foster2018uniform,
  title={Uniform convergence of gradients for non-convex learning and optimization},
  author={Foster, D. J. and Sekhari, A. and Sridharan, K.},
  journal={Advances in Neural Information Processing Systems},
  volume={31},
  pages={8759--8770},
  year={2018}
}

@article{artstein1995consistency,
  title={Consistency of minimizers and the {SLLN} for stochastic programs},
  author={Artstein, Z. and Wets, R. {J-B}},
  journal={Journal of Convex Analysis},
  volume={2},
  number={1-2},
  pages={1--17},
  year={1995}
}

@article{ioffe2009invitation,
  title={An invitation to tame optimization},
  author={Ioffe, A. D.},
  journal={SIAM Journal on Optimization},
  volume={19},
  number={4},
  pages={1894--1917},
  year={2009},
  publisher={SIAM}
}

@article{xu2010uniform,
  title={Uniform exponential convergence of sample average random functions under general sampling with applications in stochastic programming},
  author={Xu, H.},
  journal={Journal of Mathematical Analysis and Applications},
  volume={368},
  number={2},
  pages={692--710},
  year={2010},
  publisher={Elsevier}
}

@article{burke2020subdifferential,
  title={The subdifferential of measurable composite max integrands and smoothing approximation},
  author={Burke, J. V. and Chen, X. and Sun, H.},
  journal={Mathematical Programming},
  volume={181},
  number={2},
  pages={229--264},
  year={2020},
  publisher={Springer}
}

@article{qi2022asymptotic,
  title={Asymptotic properties of stationary solutions of coupled nonconvex nonsmooth empirical risk minimization},
  author={Qi, Z. and Cui, Y. and Liu, Y. and Pang, J.-S.},
  journal={Mathematics of Operations Research},
  volume={47},
  number={3},
  pages={2034--2064},
  year={2022},
  publisher={INFORMS}
}

@article{milz2025consistency,
  title={Consistency of sample-based stationary points for infinite-dimensional stochastic optimization},
  author={Milz, J.},
  journal={SIAM Journal on Optimization},
  volume={35},
  number={1},
  pages={42--61},
  year={2025},
  publisher={SIAM}
}

@article{davis2020stochastic,
  title={Stochastic subgradient method converges on tame functions},
  author={Davis, D. and Drusvyatskiy, D. and Kakade, S. and Lee, J. D.},
  journal={Foundations of Computational Mathematics},
  volume={20},
  number={1},
  pages={119--154},
  year={2020},
  publisher={Springer}
}

@article{bolte2021conservative,
  title={Conservative set valued fields, automatic differentiation, stochastic gradient methods and deep learning},
  author={Bolte, J. and Pauwels, E.},
  journal={Mathematical Programming},
  volume={188},
  number={1},
  pages={19--51},
  year={2021},
  publisher={Springer}
}

@book{shalev2014understanding,
  title={Understanding Machine Learning: From Theory to Algorithms},
  author={Shalev-Shwartz, S. and Ben-David, S.},
  year={2014},
  publisher={Cambridge University Press}
}

@book{cohn2013measure,
  title={Measure Theory},
  author={Cohn, D. L.},
  series={Birkh\"auser Advanced Texts Basler Lehrb\"ucher},
  edition = {2nd},
  year={2013},
  publisher={Springer}
}

@book{folland1999real,
  title={Real Analysis: Modern Techniques and Their Applications},
  author={Folland, G. B.},
  edition = {2nd},
  year={1999},
  publisher={John Wiley \& Sons}
}

@book{chang1990model,
  title={Model Theory},
  author={Chang, C. C. and Keisler, H. J.},
  volume={73},
  edition = {3rd},
  year={1990},
  publisher={Elsevier}
}

@article{bolte2023subgradient,
  title={Subgradient sampling for nonsmooth nonconvex minimization},
  author={Bolte, J. and Le, T. and Pauwels, E.},
  journal={SIAM Journal on Optimization},
  volume={33},
  number={4},
  pages={2542--2569},
  year={2023},
  publisher={SIAM}
}

@article{kleywegt2002sample,
  title={The sample average approximation method for stochastic discrete optimization},
  author={Kleywegt, A. J. and Shapiro, A. and {Homem-de-Mello}, T.},
  journal={SIAM Journal on Optimization},
  volume={12},
  number={2},
  pages={479--502},
  year={2002},
  publisher={SIAM}
}

@article{shapiro2000rate,
  title={On the rate of convergence of optimal solutions of {M}onte {C}arlo approximations of stochastic programs},
  author={Shapiro, A. and {Homem-de-Mello}, T.},
  journal={SIAM Journal on Optimization},
  volume={11},
  number={1},
  pages={70--86},
  year={2000},
  publisher={SIAM}
}

@article{iusem2010distances,
  title={Distances between closed convex cones: Old and new results},
  author={Iusem, A. and Seeger, A.},
  journal={Journal of  Convex Analysis},
  volume={17},
  number={3-4},
  pages={1033--1055},
  year={2010}
}

@article{kaiser2012first,
  title={First order tameness of measures},
  author={Kaiser, T.},
  journal={Annals of Pure and Applied Logic},
  volume={163},
  number={12},
  pages={1903--1927},
  year={2012},
  publisher={Elsevier}
}

@article{bareilles2025deep,
  title={Deep Learning as the Disciplined Construction of Tame Objects},
  author={Bareilles, G. and Gehret, A. and Aspman, J. and Lep{\v{s}}ov{\'a}, J. and Mare{\v{c}}ek, J.},
  journal={arXiv preprint arXiv:2509.18025},
  year={2025}
}

@book{aliprantis2006infinite,
  title={Infinite Dimensional Analysis: A Hitchhiker's Guide},
  author={Aliprantis, C. D. and Border, K. C.},
  year={2006},
    edition = {3rd},
  publisher={Springer}
}

@incollection{van2000preservation,
  title={Preservation theorems for {G}livenko-{C}antelli and uniform {G}livenko-{C}antelli classes},
  author={van der Vaart, A. W. and Wellner, J. A.},
  booktitle={High Dimensional Probability II},
  pages={115--133},
  year={2000},
  publisher={Springer}
}

@article{berend2012missing,
  title={The missing mass problem},
  author={Berend, D. and Kontorovich, A.},
  journal={Statistics \& Probability Letters},
  volume={82},
  number={6},
  pages={1102--1110},
  year={2012},
  publisher={Elsevier}
}

@InProceedings{devale2025uniform,
  title={Uniform Convergence Beyond {G}livenko-{C}antelli},
  author={Devale, T. and Devulapalli, P. and Hanneke, S.},
  booktitle = 	 {International Conference on Algorithmic Learning Theory},
  pages = 	 {1--21},
  year = 	 {2026},
  editor = 	 {Telgarsky, Matus and Ullman, Jonathan},
  volume = 	 {313},
  series = 	 {Proceedings of Machine Learning Research},
  month = 	 {23--26 Feb},
  publisher =    {PMLR},
}

@article{schechtman2026gradient,
  title={The gradient's limit of a definable family of functions admits a variational stratification},
  author={Schechtman, S.},
  journal={SIAM Journal on Optimization},
volume = {36},
number = {2},
pages = {1075-1099},
  year={2026}
}

@article{gray2002quantization,
  title={Quantization},
  author={Gray, R. M. and Neuhoff, D. L.},
  journal={IEEE Transactions on Information Theory},
  volume={44},
  number={6},
  pages={2325--2383},
  year={1998},
  publisher={IEEE}
}

@incollection{steinhorn1985chapter,
  title={Chapter {XVI}: {B}orel Structures and Measure and Category Logics},
  author={Steinhorn, C. I.},
  booktitle={Model-Theoretic Logics},
  volume={8},
  pages={579--596},
  year={2017},
  publisher={Cambridge University Press}
}

@article{bolte2007clarke,
  title={Clarke subgradients of stratifiable functions},
  author={Bolte, J. and Daniilidis, A. and Lewis, A. and Shiota, M.},
  journal={SIAM Journal on Optimization},
  volume={18},
  number={2},
  pages={556--572},
  year={2007},
  publisher={SIAM}
}

@article{nguyen2026rates,
  title={On rates of convergence for sample average approximations without smoothness},
  author={Nguyen, H. D. and Westerhout, J. and Guo, X.},
  journal={arXiv preprint arXiv:2604.25153},
  year={2026}
}

@article{chase2019model,
  title={Model Theory and Machine Learning},
  author={Chase, H. and Freitag, J.},
  journal={Bulletin of Symbolic Logic},
  volume={25},
  number={3},
  pages={319--332},
  year={2019},
  publisher={Cambridge University Press}
}

@article{karpinski1997polynomial,
  title={Polynomial bounds for {VC} dimension of sigmoidal and general {P}faffian neural networks},
  author={Karpinski, M. and Macintyre, A.},
  journal={Journal of Computer and System Sciences},
  volume={54},
  number={1},
  pages={169--176},
  year={1997},
  publisher={Elsevier}
}

@inproceedings{livni2013honest,
  title={Honest compressions and their application to compression schemes},
  author={Livni, R. and Simon, P.},
  booktitle={Conference on Learning Theory},
  pages={77--92},
  year={2013},
  organization={PMLR}
}

@inproceedings{macintyre1993finiteness,
  title={Finiteness results for sigmoidal ``neural'' networks},
  author={Macintyre, A. and Sontag, E. D.},
  booktitle={ACM Symposium on Theory of Computing},
  pages={325--334},
  year={1993}
}

@article{van1996geometric,
  title={Geometric categories and o-minimal structures},
  author={van den Dries, L. and Miller, C.},
  journal={Duke Mathematical Journal},
  volume={84},
  number={2},
  pages={497--540},
  year={1996},
  publisher={Duke University Press}
}

@article{van1996new,
  title={New {D}onsker classes},
  author={van der Vaart, A. W.},
  journal={Annals of Probability},
  volume={24},
  number={4},
  pages={2128--2140},
  year={1996},
  publisher={Institute of Mathematical Statistics}
}

@article{wilkie1996model,
  title={Model completeness results for expansions of the ordered field of real numbers by restricted {P}faffian functions and the exponential function},
  author={Wilkie, A. J.},
  journal={Journal of the American Mathematical Society},
  volume={9},
  number={4},
  pages={1051--1094},
  year={1996}
}

@article{shapiro2002conditioning,
  title={Conditioning of convex piecewise linear stochastic programs},
  author={Shapiro, A. and {Homem-de-Mello}, T. and Kim, J.},
  journal={Mathematical Programming},
  volume={94},
  number={1},
  pages={1--19},
  year={2002},
  publisher={Springer}
}

@article{burke1993weak,
  title={Weak sharp minima in mathematical programming},
  author={Burke, J. V. and Ferris, M. C.},
  journal={SIAM Journal on Control and Optimization},
  volume={31},
  number={5},
  pages={1340--1359},
  year={1993},
  publisher={SIAM}
}

@incollection{rockafellar1982favorable,
  author    = {Rockafellar, R. T.},
  title     = {Favorable Classes of {L}ipschitz-Continuous Functions in Subgradient Optimization},
  booktitle = {Progress in Nondifferentiable Optimization},
  editor    = {Nurminski, E.},
  series    = {IIASA Collaborative Proceedings Series},
  publisher = {International Institute for Applied Systems Analysis},
  address   = {Laxenburg, Austria},
  year      = {1982},
  pages     = {125--144}
}

@article{areces2026finding,
  title={Finding a stationary point of a stochastic convex problem},
  author={Areces, F. and Duchi, J. and Sommers, M.},
  journal={arXiv preprint arXiv:2607.06883},
  year={2026}
}

@book{coste2002introduction,
  title={An Introduction to Semialgebraic Geometry},
  author={Coste, M.},
    series={RAAG Notes},
    publisher={Institut de Recherche Math\'ematiques de Rennes},
  year={2002}
}

@article{aumann1965integrals,
  title={Integrals of set-valued functions},
  author={Aumann, R. J.},
  journal={Journal of Mathematical Analysis and Applications},
  volume={12},
  number={1},
  pages={1--12},
  year={1965},
  publisher={Academic Press}
}

@book{beer2023bornologies,
  title={Bornologies and Lipschitz Analysis},
  author={Beer, G.},
  year={2023},
  publisher={CRC Press}
}

@book{coste1999introduction,
  title={An Introduction to O-minimal Geometry},
  author={Coste, M.},
     series={RAAG Notes},
    publisher={Institut de Recherche Math\'ematique de Rennes},
  year={1999}
}

@article{daniilidis2020pathological,
  title={Pathological subgradient dynamics},
  author={Daniilidis, A. and Drusvyatskiy, D.},
  journal={SIAM Journal on Optimization},
  volume={30},
  number={2},
  pages={1327--1338},
  year={2020},
  publisher={SIAM}
}

@article{stability11duke,
author = {Cluckers, R. and Miller, D. J.},
title = {{Stability under integration of sums of products of real globally subanalytic functions and their logarithms}},
volume = {156},
journal = {Duke Mathematical Journal},
number = {2},
publisher = {Duke University Press},
pages = {311--348},
year = {2011}
}

@article{ruszczynski2021stochastic,
  title={A stochastic subgradient method for nonsmooth nonconvex multilevel composition optimization},
  author={Ruszczy\'nski, A.},
  journal={SIAM Journal on Control and Optimization},
  volume={59},
  number={3},
  pages={2301--2320},
  year={2021},
  publisher={SIAM}
}

@article{goldstein1977optimization,
  title={Optimization of {L}ipschitz continuous functions},
  author={Goldstein, A. A.},
  journal={Mathematical Programming},
  volume={13},
  number={1},
  pages={14--22},
  year={1977},
  publisher={Springer}
}

@article{burke2020gradient,
  title={Gradient sampling methods for nonsmooth optimization},
  author={Burke, J. V. and Curtis, F. E. and Lewis, A. S. and Overton, M. L. and Sim{\~o}es, L. E. A.},
  journal={Numerical Nonsmooth Optimization: State of the art algorithms},
  pages={201--225},
  year={2020},
  publisher={Springer}
}

@article{le2026graphical,
  title={Graphical stability of set-valued integrals under measure perturbations},
  author={Le, T.},
  journal={arXiv preprint arXiv:2607.19035},
  year={2026}
}

@article{xu2009smooth,
  title={Smooth sample average approximation of stationary points in nonsmooth stochastic optimization and applications},
  author={Xu, H. and Zhang, D.},
  journal={Mathematical Programming},
  volume={119},
  number={2},
  pages={371--401},
  year={2009},
  publisher={Springer}
}

@article{davis2022proximal,
  title={Proximal methods avoid active strict saddles of weakly convex functions},
  author={Davis, D. and Drusvyatskiy, D.},
  journal={Foundations of Computational Mathematics},
  volume={22},
  number={2},
  pages={561--606},
  year={2022},
  publisher={Springer}
}

@article{lemarechal2000,
  title={The $\mathcal{U}$-{L}agrangian of a convex function},
  author={Lemar{\'e}chal, C. and Oustry, F. and Sagastiz{\'a}bal, C.},
  journal={Transactions of the American Mathematical Society},
  volume={352},
  number={2},
  pages={711--729},
  year={2000}
}

@article{lewis2002active,
  title={Active sets, nonsmoothness, and sensitivity},
  author={Lewis, A. S.},
  journal={SIAM Journal on Optimization},
  volume={13},
  number={3},
  pages={702--725},
  year={2002},
  publisher={SIAM}
}

@book{hiriart2004fundamentals,
  title={Fundamentals of Convex Analysis},
  author={Hiriart-Urruty, J.-B. and Lemar{\'e}chal, C.},
  year={2004},
  publisher={Springer Science \& Business Media}
}

@book{primer,
	Author = {J. O. Royset and R. J-B Wets},
	Edition = {},
	Publisher = {Springer},
	Series = {},
	Title = {An Optimization Primer},
	Volume = {},
	Year = {2021}}

@article{banholzer2022rates,
  title={On rates of convergence for sample average approximations in the almost sure sense and in mean},
  author={Banholzer, D. and Fliege, J. and Werner, R.},
  journal={Mathematical Programming},
  volume={191},
  number={1},
  pages={307--345},
  year={2022},
  publisher={Springer}
}
\bibliographystyle{plainnat}

\end{document}